\documentclass[11pt]{article}
\usepackage[margin=1in]{geometry} %margins
\usepackage{amsmath, amsthm, amsfonts, amssymb} %ams packages
\usepackage{mathrsfs} %script letters
\usepackage{enumerate} %custom lists
\usepackage{xcolor} %allow colorful text
\usepackage{bbm} %additional blackboard bold characters
\usepackage{hyperref} %include hyperlinks

\usepackage{thmtools}
\usepackage{thm-restate}
\usepackage{cleveref}

\theoremstyle{theorem}
	
	\newtheorem{thm}{Theorem}[section]
	\newtheorem{lem}[thm]{Lemma}
	\newtheorem{prop}[thm]{Proposition}
	\newtheorem{cor}[thm]{Corollary}
	\newtheorem{conj}[thm]{Conjecture}
\theoremstyle{definition}
	\newtheorem{defn}[thm]{Definition}
	\newtheorem{rem}[thm]{Remark}

\newcommand{\N}{\mathbb{N}}
\newcommand{\Z}{\mathbb{Z}}
\newcommand{\Q}{\mathbb{Q}}
\newcommand{\R}{\mathbb{R}}
\newcommand{\C}{\mathbb{C}}
\newcommand{\T}{\mathbb{T}}

\newcommand{\F}{\mathbb{F}}

\renewcommand{\O}{\mathcal{O}}
\renewcommand{\P}{\mathcal{P}}

\newcommand{\B}{\mathcal{B}}
\newcommand{\D}{\mathcal{D}}

\newcommand{\calZ}{\mathcal{Z}}
\newcommand{\calK}{\mathcal{K}}

\newcommand{\calF}{\mathcal{F}}

\newcommand{\es}{\emptyset}

\newcommand{\eps}{\varepsilon}

\newcommand{\ind}{\mathbbm{1}}

\renewcommand{\hat}{\widehat}
\renewcommand{\tilde}{\widetilde}

\newcommand{\norm}[2]{\left\| #2 \right\|_{#1}}

\newcommand{\E}[2]{\mathbb{E}\left[ {#1} \mid {#2} \right]}

\newcommand{\X}{\textbf{X}}
\newcommand{\Y}{\textbf{Y}}

\newcommand{\UClim}{\text{UC-}\lim}

\newcommand{\spn}{\text{span}}

\title{Multiple recurrence and popular differences for polynomial patterns in rings of integers}
\author{Ethan Ackelsberg and Vitaly Bergelson}

\begin{document}

\maketitle

\begin{abstract}

We demonstrate that the phenomenon of popular differences (aka the phenomenon of large intersections)
holds for natural families of polynomial patterns in rings of integers of number fields.
If $K$ is a number field with ring of integers $\O_K$ and $E \subseteq \O_K$ has
positive upper Banach density $d^*(E) = \delta > 0$, we show, \emph{inter alia}:

\begin{enumerate}[1.]
	\item	If $p(x) \in K[x]$ is an \emph{intersective} polynomial
		(i.e., $p$ has a root modulo $m$ for every $m \in \O_K$) with $p(\O_K) \subseteq \O_K$
		and $r, s \in \O_K$ are distinct and nonzero, then for any $\eps > 0$,
		there is a syndetic set $S \subseteq \O_K$ such that for any $n \in S$,
		\begin{equation*}
			d^* \left( \left\{ x \in \O_K : \{x, x + rp(n), x + sp(n)\} \subseteq E \right\} \right) > \delta^3 - \eps.
		\end{equation*}
		Moreover, if $\frac{s}{r} \in \Q$, then there are syndetically many $n \in \O_K$ such that
		\begin{equation*}
			 d^* \left( \left\{ x \in \O_K : \{x, x + rp(n), x + sp(n), x + (r+s)p(n)\} \subseteq E \right\} \right)
			 > \delta^4 - \eps.
		\end{equation*}
	\item	If $\{p_1, \dots, p_k\} \subseteq K[x]$ is a \emph{jointly intersective} family
		(i.e., $p_1, \dots, p_k$ have a common root modulo $m$ for every $m \in \O_K$)
		of linearly independent polynomials with $p_i(\O_K) \subseteq \O_K$,
		then there are syndetically many $n \in \O_K$ such that
		\begin{equation*}
			d^* \left( \left\{ x \in \O_K : \{x, x + p_1(n), \dots, x + p_k(n)\} \subseteq E \right\} \right)
			 > \delta^{k+1} - \eps.
		\end{equation*}
\end{enumerate}

These two results generalize and extend previous work of Frantzikinakis and Kra \cite{frakra2} and Franztikinakis \cite{fra}
on polynomial configurations in $\Z$
and build upon recent work of the authors and Best \cite{abb} on linear patterns in general abelian groups.
The above combinatorial results follow from multiple recurrence results in ergodic theory
via a version of Furstenberg's correspondence principle.
The ergodic-theoretic recurrence theorems require a sharpening of existing tools
for handling polynomial multiple ergodic averages.
A key advancement made in this paper is a new result on the equidistribution of polynomial orbits in nilmanifolds,
which can be seen as a far-reaching generalization of Weyl's equidistribution theorem for polynomials of several variables:

\begin{enumerate}[3.]
	\item Let $d, k, l \in \N$.
	Let $(X, \B, \mu, T_1, \dots, T_l)$ be an ergodic, connected $\Z^l$-nilsystem.
	Let $\{p_{i,j} : 1 \le i \le k, 1 \le j \le l\} \subseteq \Q[x_1, \dots, x_d]$ be a family of polynomials
	such that $p_{i,j}\left( \Z^d \right) \subseteq \Z$ and $\{1\} \cup \{p_{i,j}\}$ is linearly independent over $\Q$.
	Then the $\Z^d$-sequence
	$\left( \prod_{j=1}^l{T_j^{p_{1,j}(n)}}x, \dots, \prod_{j=1}^l{T_j^{p_{k,j}(n)}}x \right)_{n \in \Z^d}$
	is well-distributed in $X^k$ for every $x$ in a co-meager set of full measure.
\end{enumerate}

\end{abstract}

%%%%%%%%%%%%%%%%%%%%%%%%%%%%%%%%%%%%%%%%%%%%%%%%%%%%
%%%%%%%%%%%%%%%%%%%%%%%%%%%%%%%%%%%%%%%%%%%%%%%%%%%%

\section{Introduction}

%%%%%%%%%%%%%%%%%%%%%%%%%%%%%%%%%%%%%%%%%%%%%%%%%%%%

\subsection{Background and main results}

Let $(X, \B, \mu, T)$ be an invertible probability measure-preserving system.
A classical result of Khintchine \cite{kh} says that for any $A \in \B$,
\begin{equation*}
	\lim_{N - M \to \infty}{\frac{1}{N-M} \sum_{n=M}^{N-1}{\mu(A \cap T^{-n}A)}} \ge \mu(A)^2.
\end{equation*}
As a consequence, for any $\eps > 0$, the set
\begin{equation*}
	R = \left\{ n \in \Z : \mu(A \cap T^{-n}A) > \mu(A)^2 - \eps \right\}
\end{equation*}
is \emph{syndetic}, meaning that it has bounded gaps (equivalently, finitely many translates of $R$ cover $\Z$).
Furstenberg showed in \cite{diag} that for any $A \in \B$ and any $k \in \N$,
\begin{equation*}
	\liminf_{N - M \to \infty}{\frac{1}{N-M} \sum_{n=M}^{N-1}{\mu(A \cap T^{-n}A \cap \dots \cap T^{-kn}A)}} > 0,
\end{equation*}
from which it follows that
\begin{equation*}
	\left\{ n \in \Z : \mu(A \cap T^{-n}A \cap \dots \cap T^{-kn}A) > c \right\}
\end{equation*}
is syndetic for some $c > 0$.
One may ask, for these longer expressions, if $c$ can be made arbitrarily close to $\mu(A)^{k+1}$.
(By considering weakly mixing systems, it is clear that $c$ cannot exceed $\mu(A)^{k+1}$ in general.)
A somewhat surprising answer was given in \cite{bhk}:

\begin{thm}[\cite{bhk}, Theorems 1.2 and 1.3] \label{thm: BHK}~
	\begin{enumerate}[1.]
		\item For any ergodic invertible probability measure-preserving system $(X, \B, \mu,T)$, any $\eps > 0$,
			and any $A \in \B$, the set
			\begin{equation} \label{eq: 3-AP large int}
				\left\{ n \in \Z : \mu(A \cap T^{-n}A \cap T^{-2n}A) > \mu(A)^3 - \eps \right\}
			\end{equation}
			is syndetic.
		\item	For any ergodic invertible probability measure-preserving system $(X, \B, \mu,T)$, any $\eps > 0$,
			and any $A \in \B$, the set
			\begin{equation*}
				\left\{ n \in \Z : \mu(A \cap T^{-n} A \cap T^{-2n}A \cap T^{-3n}A) > \mu(A)^4 - \eps \right\}
			\end{equation*}
			is syndetic.
		\item	There exists an ergodic system $(X, \B, \mu, T)$ with the following property:
			for any integer $l \geq 1$, there is a set $A =A(l) \in \B$ of positive measure such that
			\begin{equation*}
				\mu(A \cap T^{-n}A \cap T^{-2n} A \cap T^{-3n}A\cap T^{-4n}A ) \leq \frac 12\mu(A)^l
			\end{equation*}
			for every integer $n \neq 0$.
	\end{enumerate}
\end{thm}

In the terminology of \cite{abb}, Theorem \ref{thm: BHK} shows that the families $\{n, 2n\}$
and $\{n, 2n, 3n\}$ have the \emph{large intersections property}, while $\{n, 2n, \dots, kn\}$
does not have the large intersections property for $k \ge 4$.
The combinatorial content, via Furstenberg's correspondence principle,
is that, for arithmetic progression of length 3 and 4, one can find a ``popular'' common difference:
if $E \subseteq \Z$ has positive upper Banach density $d^*(E) = \delta > 0$ and $\eps > 0$,
then there exists (syndetically many) $n \ne 0$ such that
\begin{equation*}
	d^* \left( \left\{ x \in \Z : \{x, x + n, x + 2n\} \subseteq E \right\} \right) > \delta^3 - \eps,
\end{equation*}
and there exists (syndetically many) $m \ne 0$ such that
\begin{equation*}
	d^* \left( \left\{ x \in \Z : \{x, x + m, x + 2m, x + 3m\} \subseteq E \right\} \right) > \delta^4 - \eps.
\end{equation*}
A natural question to ask is whether various extensions of Szemer\'{e}di's theorem
also admit large intersections variants.

The polynomial Szemer\'{e}di theorem of the second author and Leibman \cite{bl-poly}
extends Furstenberg's result to polynomial configurations.
We say that a polynomial $p(x) \in \Q[x]$ is \emph{integer-valued} if $p(\Z) \subseteq \Z$.

\begin{thm}[\cite{bl-poly}, special case of Theorem A] \label{thm: poly Sz}
	Let $p_1, \dots, p_k \in \Q[x]$ be integer-valued polynomials with zero constant term.
	Then for any invertible probability measure-preserving system
	$(X, \B, \mu, T)$ and any $A \in \B$ with $\mu(A) > 0$,
	there exists $c > 0$ such that the set
	\begin{equation} \label{eq: poly mult rec}
		R := \left\{ n \in \Z : \mu \left( A \cap T^{-p_1(n)}A \cap \dots \cap T^{-p_k(n)}A \right) > c \right\}
	\end{equation}
	has positive lower density, i.e. $\liminf_{N \to \infty}{\frac{|R \cap \{1, \dots, N\}|}{N}} > 0$.
\end{thm}

\noindent The conclusion of Theorem \ref{thm: poly Sz} was strengthened in \cite[Theorem 0.1]{bm96},
where it was shown that $R$ is syndetic for some $c > 0$ depending on $A$.

There is a wider variety of combinatorial configurations in play when polynomials are introduced,
and there is not yet a full classification of which families of polynomials have the large intersections property.
However, large intersections variants of the polynomial Szemer\'{e}di theorem
are known for two natural classes of polynomial configurations:
independent polynomials and polynomials that are integer multiples of a fixed polynomial (for $k=2, 3$).
This is summarized by the following two results, which we seek to extend in this paper:

\begin{thm}[\cite{frakra2}, Theorem 1.3] \label{thm: frakra}
	Let $p_1, \dots, p_k \in \Q[x]$ be linearly independent integer-valued polynomials with zero constant term.
	Then for any invertible probability measure-preserving system, any $A \in \B$, and any $\eps > 0$, the set
	\begin{equation*}
		\left\{ n \in \Z : \mu \left( A \cap T^{-p_1(n)}A \cap \dots \cap T^{-p_k(n)}A \right) > \mu(A)^{k+1} - \eps \right\}
	\end{equation*}
	is syndetic.
\end{thm}

\begin{thm}[\cite{fra}, Theorem C] \label{thm: fra three}
	Let $p \in \Q[x]$ be an integer-valued polynomial with zero constant term,
	and let $a, b \in \Z$ be nonzero and distinct.
	Then for any ergodic invertible probability measure-preserving system, any $A \in \B$, and any $\eps > 0$, the sets
	\begin{equation*}
		\left\{ n \in \Z : \mu \left( A \cap T^{-ap(n)}A \cap T^{-bp(n)}A \right) > \mu(A)^3 - \eps \right\}
	\end{equation*}
	and
	\begin{equation*}
		\left\{ n \in \Z : \mu \left( A \cap T^{-ap(n)}A \cap T^{-bp(n)}A \cap T^{-(a+b)p(n)}A \right)
		 > \mu(A)^4 - \eps \right\}
	\end{equation*}
	are syndetic.
\end{thm}

We have so far stated all results about polynomial multiple recurrence only for polynomials with zero constant term.
The essential feature of such families of polynomials is that they avoid ``local obstructions.''
To be precise, we say that a family of polynomials $\{p_1, \dots, p_k\}$ is \emph{jointly intersective}
if for every $m \in \N$, there exists $n \in \Z$ such that $p_i(n) \in m\Z$ for every $i = 1, \dots, k$.
If a family of polynomials is not jointly intersective, then the set appearing in \eqref{eq: poly mult rec}
will be trivial for some rotations on finitely many points.
In \cite{bll-poly}, it was shown that there are no other obstacles to multiple recurrence:

\begin{thm}[\cite{bll-poly}, Theorem 1.1]
	For a family of integer-valued polynomials $\P = \{p_1, \dots, p_k\} \subseteq \Q[x]$, the following are equivalent:
	\begin{enumerate}[(i)]
		\item	$\P$ is jointly intersective;
		\item	For any probability measure-preserving system $(X, \B, \mu, T)$ and any $A \in \B$,
			there exists $c > 0$ such that
			\begin{equation*}
				\left\{ n \in \Z : \mu \left( A \cap T^{-p_1(n)}A \cap \dots \cap T^{-p_k(n)}A \right) > c \right\}
			\end{equation*}
			is syndetic.
	\end{enumerate}
\end{thm}

The proofs of Theorems \ref{thm: frakra} and \ref{thm: fra three} can also be easily modified
to apply to families of jointly intersective polynomials. \\

The polynomial Szemer\'{e}di theorem is in fact known for polynomials of several variables with zero constant term
(see \cite[Theorem A]{bl-poly} for the result with positive lower density and \cite[Theorem 0.7]{bm} for syndeticity).
For polynomials arising from rings of integers, the polynomial Szemer\'{e}di theorem holds
for all jointly intersective polynomials.
We now make this result precise.
Fix a number field $K$ and denote by $\O_K$ its ring of integers.
By an $\O_K$-system, we will mean a quadruple $(X, \B, \mu, T)$,
where $T$ is a measure-preserving action of $(\O_K,+)$ on a probability space $(X, \B, \mu)$.

\begin{defn}
	A family of $\O_K$-valued polynomials $\{p_1, \dots, p_k\}$ is \emph{jointly intersective} if
	for every finite index subgroup $\Lambda \subseteq (\O_K, +)$,
	there exists $\xi \in \O_K$ such that $\{p_1(\xi), \dots, p_k(\xi)\} \subseteq \Lambda$.
\end{defn}

Recall that in an abelian group $G$, a set $E \subseteq G$ is \emph{syndetic} if finitely many translates of $E$ cover $G$.
That is, $G = \bigcup_{i=1}^m{(E + g_i)}$ for some $g_1, \dots, g_m \in G$.

\begin{thm}[\cite{br}, Theorem 1.6] \label{thm: br}
	Let $K$ be a number field with ring of integers $\O_K$.
	Let $p_1, \dots, p_k \in \O_K[x]$ be jointly intersective polynomials.
	For any $\O_K$-system $(X, \B, \mu, T)$
	and any $A \in \B$, there exists $c > 0$ such that the set
	\begin{equation} \label{eq: ROI poly Sz}
		\left\{ n \in \O_K : \mu \left( A \cap T^{-p_1(n)}A \cap \dots \cap T^{-p_k(n)}A \right) > c \right\}
	\end{equation}
	is syndetic.
\end{thm}

\noindent It is therefore natural to ask whether Khintchine-type recurrence theorems hold
for polynomial configurations in rings of integers.
That is, under what conditions on the polynomials $\{p_1, \dots, p_k\}$
can the constant $c$ in \eqref{eq: ROI poly Sz} be made arbitrarily close to $\mu(A)^{k+1}$?

In this paper, we provide an answer to this question in natural and important cases
by proving extensions of Theorems \ref{thm: frakra} and \ref{thm: fra three}.

\begin{restatable}{mainthm}{IndLargeInt} \label{thm: IndLargeInt}
	Let $K$ be a number field with ring of integers $\O_K$.
	Suppose $\{p_1, \dots, p_k\} \subseteq K[x]$ is a jointly intersective family
	of linearly independent $\O_K$-valued polynomials.
	Then for any measure-preserving $\O_K$-system $\left( X, \B, \mu, T \right)$, $A \in \B$, and $\eps > 0$, the set
	\begin{equation} \label{eq: IndLargeInt}
		\left\{ n \in \O_K :
		 \mu \left( A \cap T^{-p_1(n)}A \cap \cdots \cap T^{-p_k(n)}A \right) > \mu(A)^{k+1} - \eps \right\}
	\end{equation}
	is syndetic.
\end{restatable}

\begin{restatable}{mainthm}{MultLargeInt} \label{thm: MultLargeInt}
	Let $K$ be a number field with ring of integers $\O_K$.
	Let $p(x) \in K[x]$ be an $\O_K$-valued intersective polynomial.
	Let $r,s \in \O_K$ be distinct and nonzero.
	Then for any ergodic measure-preserving $\O_K$-system $\left( X, \B, \mu, T \right)$,
	$A \in \B$, and $\eps > 0$, the set
	\begin{equation} \label{eq: MultLargeInt triple}
		\left\{ n \in \O_K :
		 \mu \left( A \cap T^{-rp(n)}A \cap T^{-sp(n)}A \right) > \mu(A)^3 - \eps \right\}
	\end{equation}
	is syndetic.
	
	Moreover, if $\frac{s}{r} \in \Q$, then
	\begin{equation} \label{eq: MultLargeInt quadruple}
		\left\{ n \in \O_K :
		 \mu \left( A \cap T^{-rp(n)}A \cap T^{-sp(n)}A \cap T^{-(r+s)p(n)}A \right) > \mu(A)^4 - \eps \right\}
	\end{equation}
	is syndetic.
\end{restatable}

Note that for a pair of polynomials $\{p,q\} \subseteq K[x] \setminus \{0\}$,
either $p$ and $q$ are linearly independent over $K$ or $q = cp$ for some $c \in K$.
Thus, we have the following immediate consequence of
Theorems \ref{thm: IndLargeInt} and \ref{thm: MultLargeInt} together:

\begin{cor}
	Let $K$ be a number field with ring of integers $\O_K$.
	Suppose $\{p, q\} \subseteq K[x]$ is a jointly intersective pair of $\O_K$-valued polynomials.
	Then for any ergodic measure-preserving $\O_K$-system $\left( X, \B, \mu, T \right)$,
	any $A \in \B$, and any $\eps > 0$, the set
	\begin{equation*}
		\left\{ n \in \O_K :
		 \mu \left( A \cap T^{-p(n)}A \cap T^{-q(n)}A \right) > \mu(A)^3 - \eps \right\}
	\end{equation*}
	is syndetic.
\end{cor}

Theorem \ref{thm: IndLargeInt} shows that for independent families of \emph{any size},
we can achieve Khintchine-type results.
In contrast, Theorem \ref{thm: MultLargeInt} only demonstrates a Khintchine-type result
for configurations of length three or four and requires ergodicity of the system
(for counterexamples in the non-ergodic case, see \cite[Section 11.1]{abb}).
Moreover, for length four, we have made additional assumptions, which we discuss below.
To complete the picture, we now address what happens for patterns of length five and longer.
For concreteness, let us consider general polynomial families of the form $\{a_1p, \dots, a_kp\}$,
where $a_i \in \O_K$ and $p(x) \in K[x]$ is $\O_K$-valued.
In the simplest case when $K = \Q$ and $a_i = i$, a combinatorial construction of Ruzsa
rules out Khintchine-type results when $k \ge 4$ (see item 3 of Theorem \ref{thm: BHK} above).
In \cite[Corollary 12.14]{abb}, this was generalized to any number field $K$ and any integers $a_i \in \Z$ for $k \ge 4$.
Furthermore, \cite[Proposition 12.13]{abb} gives a combinatorial criterion for checking the case
$k = 4$ for any coefficients $a_i \in \O_K$.
We do not know how to prove the requisite combinatorial result, but we believe that
Khintchine-type results will fail for any non-trivial family $\{a_1p, \dots, a_kp\}$ with $k \ge 4$.

Now we turn to the other conditions imposed for the patterns of length four appearing in Theorem \ref{thm: MultLargeInt}.
The strategy of proof in Theorem \ref{thm: MultLargeInt} is to reduce to the linear case $p(n) = n$
and then apply knowledge about linear patterns.
General Khintchine-type results for linear patterns appear in \cite{abb} (subsequently improved in \cite{abs} and \cite{a}),
where a similar distinction is made between patterns of length three and of length four:

\begin{thm}[\cite{abb}, Theorems 1.10 and 1.11] \label{thm: ABB}
	Let $(G,+)$ be a countable discrete abelian group.
	Let $\left( X, \B, \mu, (T_g)_{g \in G} \right)$ be an ergodic measure-preserving $G$-system.
	Let $A \in \B$ and $\eps > 0$.
	
	\begin{enumerate}[1.]
		\item	Suppose $\varphi, \psi : G \to G$ are homomorphisms such that the subgroups
			$\varphi(G)$, $\psi(G)$, and $(\psi-\varphi)(G)$ have finite index in $G$.
			Then
			\begin{equation*}
				\left\{ g \in G : \mu \left( A \cap T_{\varphi(g)}^{-1}A \cap T_{\psi(g)}^{-1}A \right)
				 > \mu(A)^3 - \eps \right\}
			\end{equation*}
			is syndetic in $G$.
		\item	Suppose $r, s \in \Z$ are distinct and nonzero such that the subgroups
			$rG$, $sG$, $(r+s)G$, and $(s-r)G$ have finite index in $G$.
			Then
			\begin{equation*}
				\left\{ g \in G : \mu \left( A \cap T_{rg}^{-1}A \cap T_{sg}^{-1}A \cap T_{(r+s)g}^{-1}A \right)
				 > \mu(A)^4 - \eps \right\}
			\end{equation*}
			is syndetic in $G$.
	\end{enumerate}
\end{thm}

\noindent The second half of Theorem \ref{thm: ABB} was also proved independently in \cite[Theorem 1.3]{shalom}.
By absorbing a constant into the polynomial $p$ in Theorem \ref{thm: MultLargeInt},
imposing the condition $\frac{s}{r} \in \Q$ is equivalent to assuming $r, s \in \Z$,
so our assumptions allow us to apply Theorem \ref{thm: ABB} in the linear case $p(n) = n$.

In \cite{MIT}, it was shown that, for a related finitary problem, there are automorphisms $\varphi$ and $\psi$
such that $\varphi + \psi$ and $\psi - \varphi$ are also automorphisms but for which a Khintchine-type result fails:

\begin{thm}[\cite{MIT}, Theorem 1.3] \label{thm: no pop diff}
	There is an absolute constant $c > 0$ such that the following holds.
	If $\alpha \in (0, c)$, then for all sufficiently large $n$ (depending on $\alpha$),
	there is a set $A \subseteq (\F_5^n)^2$ with $|A| \ge \alpha \cdot 5^{2n}$ such that
	\begin{equation*}
		\left| A \cap A - (a,b) \cap A - (b, -a) \cap A - (a+b, b-a) \right| \le (1-c)\alpha^4 \cdot 5^{2n}
	\end{equation*}
	for all $(a, b) \in (\F_5^n)^2 \setminus \{(0,0)\}$.
\end{thm}

\noindent The authors of \cite{MIT} explain the failure of large intersections in Theorem \ref{thm: no pop diff}
as a consequence of an eigenvalue condition.
Namely, for the corresponding matrices
\begin{equation*}
	M_1 = \left( \begin{array}{cc} 1 & 0 \\ 0 & 1 \end{array} \right) \quad \text{and} \quad
	M_2 = \left( \begin{array}{cc} 0 & -1 \\ 1 & 0 \end{array} \right),
\end{equation*}
the eigenvalues of $M_1M_2^{-1}$ are negatives of each other.
They also show that in the absence of such an eigenvalue condition,
a Khintchine-type result holds for patterns
\begin{equation*}
	\left\{ x, x + M_1y, x + M_2y, x + (M_1+M_2)y \right\}
\end{equation*}
(see \cite[Theorem 1.2]{MIT}).

In our context of rings of integers, we can translate the eigenvalue condition into an algebraic criterion.
Recall that two algebraic numbers $\alpha, \beta \in K$ are \emph{conjugate} (over $\Q$)
if they have the same minimal polynomial (over $\Q$).
Equivalently, there is a field automorphism $\varphi : K \to K$ such that $\varphi(\alpha) = \beta$.
If we denote by $M_{\alpha}$ the $\Q$-linear map $M_{\alpha}x = \alpha x$ on the $\Q$-vector space $K$,
then the eigenvalues of $M_{\alpha}$ are exactly the conjugates of $\alpha$
(this follows from, e.g., \cite[Theorem 5.9]{conrad},
which gives a formula for the characteristic polynomial of $M_{\alpha}$).
We therefore make the following conjecture:

\begin{conj}
	Let $K$ be a number field with ring of integers $\O_K$.
	Let $r, s \in \O_K$ be distinct and nonzero.
	The following are equivalent:
	
	\begin{enumerate}[(i)]
		\item	For any ergodic measure-preserving $\O_K$-system $(X, \B, \mu, T)$,
			any $A \in \B$, any $\eps > 0$, and any $\O_K$-valued intersective polynomial $p \in K[x]$, the set
			\begin{equation*}
				\left\{ n \in \O_K :
				 \mu \left( A \cap T^{-rp(n)}A \cap T^{-sp(n)}A \cap T^{-(r+s)p(n)}A \right) > \mu(A)^4 - \eps \right\}
			\end{equation*}
			is syndetic;
		\item	No two conjugates of $\frac{s}{r}$ over $\Q$ are negatives of each other.
	\end{enumerate}
\end{conj}

%%%%%%%%%%%%%%%%%%%%%%%%%%%%%%%%%%%%%%%%%%%%%%%%%%%%

\subsection{Method}

In order to prove Khintchine-type recurrence results such as
Theorem \ref{thm: IndLargeInt} and Theorem \ref{thm: MultLargeInt},
it is natural to consider associated multiple ergodic averages.
The appropriate averaging schemes in rings of integers are those arising from \emph{F{\o}lner sequences}.
A \emph{F{\o}lner sequence} in $(\O_K, +)$ is a sequence of subsets $(\Phi_N)_{N \in \N}$ of $\O_K$
such that, for every $n \in \O_K$,
\begin{equation*}
	\frac{|(\Phi_N + n) \triangle \Phi_N|}{|\Phi_N|} \to 0.
\end{equation*}
Examples of F{\o}lner sequences include boxes in $\O_K \cong \Z^d$ with increasing side lengths.
We say that a sequence $(u_n)_{n \in \O_K}$ has \emph{uniform Ces\`{a}ro limit} $u$,
denoted $\UClim_{n \in \O_K}{u_n} = u$, if
\begin{equation*}
	\frac{1}{|\Phi_N|} \sum_{n \in \Phi_N}{u_n} \to u
\end{equation*}
for every F{\o}lner sequence $(\Phi_N)_{N \in \N}$ in $(\O_K, +)$.
The usefulness of uniform Ces\`{a}ro limits in proving Khintchine-type theorems comes from the following routine fact
(for a proof, see \cite[Lemma 1.9]{abb}):

\begin{prop} \label{prop: Folner syndetic}
	A set $S \subseteq \O_K$ is syndetic if and only if for any F{\o}lner sequence $(\Phi_N)_{N \in \N}$ in $(\O_K, +)$,
	one has $\bigcup_{N \in \N}{\Phi_N} \cap S \ne \es$.
\end{prop}

Rather than computing the multiple ergodic averages
\begin{equation} \label{eq: poly mult erg avg}
	\UClim_{n \in \O_K}{\prod_{i=1}^k{T^{p_i(n)}f_i}}
\end{equation}
directly for an arbitrary $\O_K$-system, we reduce to computing the averages \eqref{eq: poly mult erg avg}
in simpler classes of systems.
To be precise, we say a system $\Y = \left( Y, \D, \nu, S \right)$ is a \emph{factor} of $\X = \left( X, \B, \mu, T \right)$
if there are full measure subsets $X_0 \subseteq X$ and $Y_0 \subseteq Y$ and a measure-preserving map
$\pi : X_0 \to Y_0$ such that $S^n\pi(x) = \pi(T^nx)$ for every $x \in X_0$, $n \in \O_K$.
There is a natural correspondece between the factor $Y$ and the $T$-invariant sub-$\sigma$-algebra $\pi^{-1}(\D)$.
This allows us to take conditional expectations, and in a standard abuse of notation,
we write $\E{f}{Y} := \E{f}{\pi^{-1}(\D)}$.
The factor $\Y$ is \emph{characteristic} for a family of sequences $\{a_1(n), \dots, a_k(n)\}, n \in \O_K$,
if for any $f_1, \dots, f_k \in L^{\infty}(\mu)$,
\begin{equation*}
	\UClim_{n \in \O_K}{\left( \prod_{i=1}^k{T^{a_i(n)}f_i} - \prod_{i=1}^k{T^{a_i(n)}\E{f_i}{Y}} \right)} = 0
\end{equation*}
in $L^2(\mu)$.

The main family of factors that we will deal with is the family of \emph{nilfactors} $(\calZ_r)_{r \in \N}$
(also called \emph{Host--Kra factors} from the work of Host and Kra on $\Z$-actions \cite{hk}).
Assume for this discussion that $T$ is an ergodic action of $\O_K$.
The factor $\calZ_r$ is defined to be the minimal factor that is characteristic for all families
$\{l_1n, \dots, l_{r+1}n\}$ with $l_1, \dots, l_{r+1} \in \O_K$ distinct and nonzero.
For our purposes, it will suffice to discuss some general properties of nilfactors.

The tower of factors $\calZ_1 \subseteq \calZ_2 \subseteq \dots$
is a sequence of compact extensions.
The first factor, $\calZ_1$, is the \emph{Kronecker factor},
which is the smallest factor for which every eigenfunction is measurable.
As a measure-preserving system, it is isomorphic to an action by rotations on a compact abelian group.
The Kronecker factor contains a subfactor that will also be of interest, namely the \emph{rational Kronecker factor},
denoted $\calK_{rat}$, which is an inverse limit of finite rotational systems
(for a more detailed discussion of the rational Kronecker factor, see Section \ref{sec: rational Kronecker}).

The higher-level nilfactors also have the structure of (inverse limits of) ``rotational'' systems
but on more complex algebraic objects.
Let $G$ be an $r$-step nilpotent Lie group and $\Gamma < G$ a co-compact discrete subgroup.
The quotient space $X = G/\Gamma$ is called an \emph{$r$-step nilmanifold}.
An \emph{$r$-step nilsystem} is a system $(X, \B, \mu, T)$,
where $X = G/\Gamma$ is an $r$-step nilmanifold, $\mu$ is the Haar probability measure on $X$,
and $T$ is an $(\O_K,+)$-action by \emph{niltranslations},
i.e. transformations of the form $x \mapsto ax$ for some $a \in G$.
The nilfactor $\calZ_r$ is an inverse limit of $r$-step nilsystems.
For $\Z$-actions, this was established by Host and Kra in \cite{hk} and independently by Ziegler in \cite{ziegler}.
For our generality of $\O_K$-systems, this follows from \cite[Theorem 4.1.2]{griesmer}.

By careful application of the van der Corput differencing trick,
one can reduce polynomial expressions to (potentially much longer) linear expressions.
This works so long as the polynomials $p_1, \dots, p_k$ are \emph{essentially distinct},
meaning that $p_j - p_i$ is non-constant for every $i \ne j$.
Hence, for any family of essentially distinct polynomial sequences $\{p_1(n), \dots, p_k(n)\}, n \in \O_K$,
there is a characteristic factor that is a nilfactor
(but the step of the nilfactor may far exceed $k-1$ in general):

\begin{thm}[cf. \cite{br}, Theorem 5.2] \label{thm: polynomial nil}
	Let $K$ be a number field with ring of integers $\O_K$.
	Suppose $\{p_1, \dots, p_k\} \subseteq K[x]$ are non-constant and essentially distinct $\O_K$-valued polynomials.
	Then there is an $r \in \N$ such that for any ergodic $\O_K$-system $(X, \B, \mu, T)$
	and any $f_1, \dots, f_k \in L^{\infty}(\mu)$,
	\begin{equation*}
		\UClim_{n \in \O_K}{\prod_{i=1}^k{T^{p_i(n)}f_i}}
		 = \UClim_{n \in \O_K}{\prod_{i=1}^k{T^{p_i(n)} \E{f_i}{\calZ_r}}}.
	\end{equation*}
	in $L^2(\mu)$.
\end{thm}

For the specific configurations appearing in Theorem \ref{thm: IndLargeInt} (independent polynomials)
and in Theorem \ref{thm: MultLargeInt} (multiples of a single polynomial),
we can control the step of the characteristic nilfactors.
In order to properly formulate our results, we need one more definition.

\begin{defn}
	A family of polynomials $\{p_1, \dots, p_k\} \subseteq K[x]$ is \emph{independent}
	if for all $(c_1, \dots, c_k) \in K^k \setminus \{0\}$, the polynomial $\sum_{i=1}^k{c_ip_i}$ is non-constant.
\end{defn}

\noindent Note that the family $\{p_1, \dots, p_k\}$ is independent if and only if
$\{1, p_1, \dots, p_k\}$ is linearly independent over $K$.
Furthermore, a jointly intersective family $\{p_1, \dots, p_k\}$ is independent if and only if it is linearly independent.

\begin{restatable}{mainthm}{rationalKronecker} \label{thm: rational Kronecker}
	Let $K$ be a number field with ring of integers $\O_K$.
	Suppose $p_1, \dots, p_k \in K[x]$ are independent and $\O_K$-valued.
	Then for any ergodic measure-preserving $\O_K$-system $\left( X, \B, \mu, T \right)$
	and any $f_1, \dots, f_k \in L^{\infty}(\mu)$,
	\begin{equation*}
		\UClim_{n \in \O_K}{\prod_{i=1}^k{T^{p_i(n)}f_i}}
		 = \UClim_{n \in \O_K}{\prod_{i=1}^k{T^{p_i(n)} \E{f_i}{\calK_{rat}}}},
	\end{equation*}
	where the limits are taken in $L^2(\mu)$.
\end{restatable}

\begin{restatable}{mainthm}{nilfactor} \label{thm: nilfactor}
	Let $K$ be a number field with ring of integers $\O_K$.
	Let $p(x) \in K[x]$ be a non-constant $\O_K$-valued polynomial.
	Then for any ergodic measure-preserving $\O_K$-system $\left( X, \B, \mu, T \right)$,
	any $l_1, \dots, l_k \in \O_K$ distinct and nonzero, and any $f_1, \dots, f_k \in L^{\infty}(\mu)$,
	\begin{equation*}
		\UClim_{n \in \O_K}{\prod_{i=1}^k{T^{l_ip(n)}f_i}}
		 = \UClim_{n \in \O_K}{\prod_{i=1}^k{T^{l_ip(n)} \E{f_i}{\calZ_{k-1}}}},
	\end{equation*}
	where the limits are taken in $L^2(\mu)$.
	Moreover, if $T$ is totally ergodic, then this limit does not depend on the polynomial $p$.
\end{restatable}

We will prove Theorems \ref{thm: rational Kronecker} and \ref{thm: nilfactor} via equidistribution results
for polynomial sequences in nilmanifolds, which are of independent interest
(see Theorem \ref{thm: top tot erg} and Proposition \ref{prop: polynomial orbit} below).
After several reductions, the main technical result in the proof of Theorem \ref{thm: rational Kronecker}
is the following far-reaching generalization of Weyl's polynomial equidistribution theorem
for families of independent polynomials in several variables:

\begin{thm}[Theorem \ref{thm: Z^l equidistribution}]
	Let $d, l, k, m \in \N$.
	Let $\{p_{i,j} : 1 \le i \le k, 1 \le j \le l\} \subseteq \Q[x_1, \dots, x_d]$ be $\Z$-valued and independent over $\Q$.
	Let $T_1, \dots, T_l : \T^m \to \T^m$ be commuting unipotent affine transformations
	generating an ergodic $\Z^l$-action.
	Then the polynomial sequence
	\begin{equation*}
		\left( \prod_{j=1}^l{T_j^{p_{1,j}(n)}}x, \dots, \prod_{j=1}^l{T_j^{p_{k,j}(n)}}x \right)_{n \in \Z^d}
	\end{equation*}
	is well-distributed in $\T^{mk}$ for all $x$ in a co-meager set of full measure.
\end{thm}

The upshot of Theorem \ref{thm: rational Kronecker} is that we may compute multiple ergodic averages
for independent polynomials by studying the corresponding averages in a finite rotational system,
where computations are much easier to carry out.
Similarly, Theorem \ref{thm: nilfactor} says that in order to compute multiple ergodic averages
for $k$ distinct multiples of a fixed polynomial, we can make use of the algebraic structure of a $(k-1)$-step nilsystem.

From here, we can follow a standard technique to deduce the corresponding Khintchine-type results.
The assumption that the families of polynomials under consideration are jointly intersective,
together with a standard approximation argument,
allows us to reduce to the case that the action $T$ is \emph{totally ergodic}, i.e. that $\calK_{rat}$ is trivial.
For independent polynomials, Theorem \ref{thm: rational Kronecker}
guarantees that
\begin{equation*}
	\UClim_{n \in \O_K}{\mu \left( A \cap T^{-p_1(n)}A \cap \dots \cap T^{-p_k(n)}A \right)} = \mu(A)^{k+1}
\end{equation*}
for totally ergodic $T$, from which Theorem \ref{thm: IndLargeInt} immediately follows.
The details of this argument are carried out in Section \ref{sec: IndLargeInt}.
When all of the polynomials involved are multiples of a fixed polynomial and $T$ is totally ergodic,
Theorem \ref{thm: nilfactor} says that the relevant multiple ergodic average can be reduced to a linear average
(corresponding to $p(n) = n$).
We are therefore able to capitalize on Khintchine-type results for linear averages (see Theorem \ref{thm: ABB} above)
and extend them to the polynomial configurations we consider in Theorem \ref{thm: MultLargeInt}.
The full details of this argument appear in Section \ref{sec: MultLargeInt}.

%%%%%%%%%%%%%%%%%%%%%%%%%%%%%%%%%%%%%%%%%%%%%%%%%%%%

\subsection{Notions of largeness}

Syndeticity is just one of many notions of largeness that naturally appear in ergodic theory and combinatorics.
While it is useful in quantifying the size of subsets, it does not have all of the properties that one may desire.
To illustrate one shortcoming of syndeticity, we return to Szemer\'{e}di's theorem.
Consider the family of sets $\mathcal{R}_k := \{R_k(\X, A) : \X = (X, \B, \mu, T)~\text{mps}, A \in \B, \mu(A) > 0\}$, where
\begin{equation*}
	R_k(\X, A) := \left\{ n \in \Z : \mu \left( A \cap T^{-n}A \cap \dots \cap T^{-kn}A \right) > 0 \right\}.
\end{equation*}
The family $\mathcal{R}_k$ has the \emph{filter property}: for any $R, S \in \mathcal{R}_k$,
we have $R \cap S \ne \es$.
Indeed, given two measure-preserving systems $\X = (X, \B, \mu, T)$ and $\Y = (Y, \D, \nu, S)$,
one can form the product system $\X \times \Y$ and easily verify that
\begin{equation*}
	R_k(\X, A) \cap R_k(\Y, B) = R_k(\X \times \Y, A \times B) \in \mathcal{R}_k.
\end{equation*}
One may hope that there is a different notion of largeness that captures this filter property.
To discuss one such notion, we introduce the class of \emph{IP sets}.

Let $(x_n)_{n \in \N}$ be a sequence in $\O_K$.
The \emph{finite sum set} associated to $(x_n)_{n \in \N}$ is the set
\begin{equation*}
	FS\left( (x_n)_{n \in \N} \right) := \left\{ \sum_{n \in F}{x_n} : F \subseteq \N~\text{is finite and nonempty} \right\}.
\end{equation*}
We say that $A \subseteq \O_K$ is an \emph{IP set} if $A \supseteq FS\left( (x_n)_{n \in \N} \right)$
for some infinite sequence $(x_n)_{n \in \N}$.
A theorem of Hindman \cite{hindman} asserts that IP sets are \emph{partition regular}:

\begin{thm}[Hindman's Theorem \cite{hindman}, Theorem 3.1] \label{thm: Hindman}
	Let $A$ be an IP set.
	If $A$ is finitely partitioned $A = \bigcup_{i=1}^r{C_i}$,
	then for some $i_0 \in \{1, \dots, r\}$, $C_{i_0}$ is an IP set.
\end{thm}

\noindent A set $E$ is $\text{IP}^*$ if $E \cap A \ne \es$ for every IP set $A$.
It follows from Theorem \ref{thm: Hindman} that $\text{IP}^*$ sets have the filter property.
From this point of view, the IP polynomial Szemer\'{e}di theorem is more satisfactory:

\begin{thm}[\cite{bm}, Theorem 0.7] \label{thm: IP poly Sz}
	Let $p_1, \dots, p_k \in \Q[x]$ be integer-valued polynomials with zero constant term.
	Then for any ergodic invertible probability measure-preserving system and any $A \in \B$, the set
	\begin{equation*}
		\left\{ n \in \Z : \mu \left( A \cap T^{-p_1(n)}A \cap \dots \cap T^{-p_k(n)}A \right) > 0 \right\}
	\end{equation*}
	is $\text{IP}^*$.
\end{thm}

\begin{rem}
	For the linear pattern $p_i(n) = in$, Theorem \ref{thm: IP poly Sz} follows from \cite[Theorem A]{FK}.
\end{rem}

When bounding the size of the intersections from below, the filter property is no longer a straightforward consequence
from considering product systems.
Furthermore, $\text{IP}^*$ turns out to be too strong of a notion of largeness.
(Indeed, in a skew-product system on the torus $\T^2$,
one can find a set $A$ for which the set \eqref{eq: 3-AP large int} fails to be $\text{IP}^*$ for small $\eps > 0$.)
However, we can use a slightly weaker notion that retains the filter property.
Define the \emph{upper Banach density} of a set $E \subseteq \O_K$ by
\begin{equation*}
	d^*(E) := \sup\left\{ \limsup_{N \to \infty}{\frac{|E \cap \Phi_N|}{|\Phi_N|}}
	 : (\Phi_N)_{N \in \N}~\text{is a F{\o}lner sequence in}~\O_K \right\}.
\end{equation*}
We say that $E$ is \emph{almost $\text{IP}^*$}, or \emph{$\text{AIP}^*$} for short,
if $E$ can be written as $E = A \setminus B$, where $A$ is an $\text{IP}^*$ set and $B$ is a set with $d^*(B) = 0$.
In \cite{br}, it was shown that the set \eqref{eq: ROI poly Sz} in Theorem \ref{thm: br}
is in fact a shift of an $\text{AIP}^*$ set.

In a similar vein, Theorem \ref{thm: frakra} was strengthened in \cite{bl-cubic}.
There, the notion of largeness used is the even stronger notion of \emph{$\text{AVIP}_0^*$},
which we define in Section \ref{sec: refinements}.

\begin{thm}[\cite{bl-cubic}, Theorem 4.2]
	Let $p_1, \dots, p_k \in \Q[x]$ be linearly independent integer-valued polynomials with zero constant term.
	Then for any ergodic invertible probability measure-preserving system, any $A \in \B$, and any $\eps > 0$, the set
	\begin{equation*}
		\left\{ n \in \Z : \mu \left( A \cap T^{-p_1(n)}A \cap \dots \cap T^{-p_k(n)}A \right) > \mu(A)^{k+1} - \eps \right\}
	\end{equation*}
	is $\text{AVIP}_0^*$.
\end{thm}

In section 5, we similarly strengthen the conclusions of
Theorem \ref{thm: IndLargeInt} and Theorem \ref{thm: MultLargeInt}.
In particular, we show that, if $T$ is ergodic, then the sets \eqref{eq: IndLargeInt}, \eqref{eq: MultLargeInt triple},
and \eqref{eq: MultLargeInt quadruple} are shifts of $\text{AVIP}_0^*$ sets
(see Theorem \ref{thm: IndLargeInt AVIP} and Theorem \ref{thm: MultLargeInt AVIP}).

%%%%%%%%%%%%%%%%%%%%%%%%%%%%%%%%%%%%%%%%%%%%%%%%%%%%

\subsection{Combinatorial applications}

We deduce several combinatorial facts from the ergodic-theoretic theorems above.
For some of the combinatorial results, we have stronger finitary versions.
For other combinatorial facts derived from ergodic-theoretic results under the assumption of ergodicity,
we cannot easily deduce finitary consequences.
This distinction arises from subtleties in Furstenberg's correspondence principle, which we discuss in more detail below.
The first version of Furstenberg's correspondence principle that we will use is as follows:

\begin{thm}[\cite{et/dp}, Theorem 4.17] \label{thm: correspondence}
	Fix a F{\o}lner sequence $\Phi = (\Phi_N)_{N \in \N}$ in $\O_K$.
	Suppose $E \subseteq \O_K$ has positive upper density along $\Phi$,
	i.e. $\overline{d}_{\Phi}(E) := \limsup_{N \to \infty}{\frac{|E \cap \Phi_N|}{|\Phi_N|}} > 0$.
	Then there exists an $\O_K$-system $(X, \B, \mu, T)$ and a set $A \in \B$ with $\mu(A) = \overline{d}_{\Phi}(E)$
	such that, for any $k \in \N$ and any $n_1, \dots, n_k \in \O_K$, one has
	\begin{equation} \label{eq: correspondence}
		\overline{d}_{\Phi} \left( \bigcap_{i=1}^k{(E - n_i)} \right) \ge \mu \left( \bigcap_{i=1}^k{T^{-n_i}A} \right).
	\end{equation}
\end{thm}

Applying Theorem \ref{thm: correspondence} directly alongside Theorem \ref{thm: IndLargeInt},
we get the following:

\begin{thm} \label{thm: IndLargeInt d-bar}
	Let $K$ be a number field with ring of integers $\O_K$.
	Suppose $\{p_1, \dots, p_k\} \subseteq K[x]$ is a jointly intersective family
	of linearly independent $\O_K$-valued polynomials.
	Fix a F{\o}lner sequence $\Phi = (\Phi_N)_{N \in \N}$
	and suppose $E \subseteq \O_K$ satisfies $\overline{d}_{\Phi}(E) > 0$.
	Then for any $\eps > 0$,
	\begin{equation*}
		\left\{ n \in \O_K : \overline{d}_{\Phi} \left( E \cap (E - p_1(n)) \cap \dots \cap (E - p_k(n)) \right)
		 > \overline{d}_{\Phi}(E)^{k+1} - \eps \right\}
	\end{equation*}
	is syndetic.
\end{thm}

Taking the natural F{\o}lner sequence $\Phi_N = \{1, \dots, N\}^d$ under the isomorphism $\O_K \cong \Z^d$,
we deduce a related finitary result:

\begin{cor} \label{cor: finitary patterns}
	Let $K$ be a degree $d$ number field with ring of integers $\O_K \cong \Z^d$.
	Suppose $\{p_1, \dots, p_k\} \subseteq K[x]$ is a jointly intersective family
	of linearly independent $\O_K$-valued polynomials.
	For any $\delta, \eps > 0$, there exists $N = N(\delta, \eps) \in \N$ such that:
	if $A \subseteq \{1, \dots, N\}^d$ with $|A| > \delta N^d$, then $A$ contains at least $(\delta^{k+1} - \eps)N^d$
	configurations of the form $\{x, x + p_1(n), \dots, x + p_k(n)\}$ for some $n \ne 0$.
\end{cor}

\begin{proof}
	Let $\delta, \eps > 0$, and suppose no such $N$ exists.
	That is, for some sequence $N_m \to \infty$, we can find sets $A_m \subseteq \{1, \dots, N\}^d$
	with $|A_m| > \delta N_m^d$ such that
	\begin{equation*}
		\left| A_m \cap (A_m - p_1(n)) \cap \dots \cap (A_m - p_k(n)) \right| \le (\delta^{k+1} - \eps) N_m^d
	\end{equation*}
	for every $n \ne 0$.
	
	By passing to a subsequence if necessary, we may assume
	\begin{equation*}
		\lim_{m \to \infty}{
		 \frac{\left| (A_{m,i_1} - n_1) \cap \dots \cap (A_{m,i_r} - n_r) \cap \{1, \dots, N_m\}^d \right|}{N_m^d}}
	\end{equation*}
	exists for all $r \in \N$, $n_1, \dots, n_r \in \O_K$, and $i_1, \dots, i_r \in \{0, 1\}$,
	where $A_{m,0} = A_m$ and $A_{m,1} = \O_K \setminus A_m$.
	Then we may define a measure on $X = \{0,1\}^{\O_K}$ by letting
	\begin{equation*}
		\mu \left( \{x \in X : x_{n_1} = i_1, \dots, x_{n_r} = i_r\} \right)
		 = \lim_{m \to \infty}{
		 \frac{\left| (A_{m,i_1} - n_1) \cap \dots \cap (A_{m,i_r} - n_r) \cap \{1, \dots, N_m\}^d \right|}{N_m^d}}
	\end{equation*}
	and extending using Kolmogorov's extension theorem.
	Note that the shift map $(T^nx)(m) = x(n+m)$ preserves the measure $\mu$.
	Taking $A = \{x \in X : x_0 = 1\}$, we have
	\begin{equation*}
		\mu(A) = \lim_{m \to \infty}{\frac{|A_m|}{N_m^d}} \ge \delta,
	\end{equation*}
	and on the other hand,
	\begin{align*}
		\mu \left( A \cap T^{-p_1(n)}A \cap \dots \cap T^{-p_k(n)}A \right)
		 & = \lim_{m \to \infty}{\frac{\left| A_m \cap (A_m - p_1(n)) \cap \dots \cap (A_m - p_k(n)) \right|}{N_m^d}} \\
		 & \le \delta^{k+1} - \eps
	\end{align*}
	for $n \ne 0$.
	This contradicts Theorem \ref{thm: IndLargeInt}.
\end{proof}

\bigskip

Note that the system in Theorem \ref{thm: correspondence} may not be ergodic.
To obtain an inequality similar to \eqref{eq: correspondence} while ensuring that the system is ergodic,
one needs to allow for replacing the density along $\Phi$ by the density along some other F{\o}lner sequence
depending on the choice of translates $n_1, \dots, n_k$.
(An example due to Hindman \cite{hindman-covering} can be used to show that, for certain sets $E$,
the measure-preserving system in the conclusion of Theorem \ref{thm: correspondence} is necessarily non-ergodic;
see the discussion following Theorem 1.3 in \cite{d*} for more detail.)
Using the notion of upper Banach density,
we can formulate an ergodic version of Furstenberg's correspondence principle:

\begin{thm} \label{thm: erg corr}
	Suppose $E \subseteq \O_K$ has positive upper Banach density.
	Then there exists an ergodic $\O_K$-system $(X, \B, \mu, T)$ and a set $A \in \B$ with $\mu(A) = d^*(E)$
	such that, for any $k \in \N$ and any $n_1, \dots, n_k \in \O_K$, one has
	\begin{equation*}
		d^* \left( \bigcap_{i=1}^k{(E - n_i)} \right) \ge \mu \left( \bigcap_{i=1}^k{T^{-n_i}A} \right).
	\end{equation*}
\end{thm}

\noindent For $\Z$-actions, Theorem \ref{thm: erg corr} appears in \cite[Proposition 3.1]{bhk},
utilizing an observation of Emmanuel Lesigne based on the original argument of Furstenberg.
For a general version in \emph{amenable} groups
(a class containing all countable abelian groups), see \cite[Theorem 2.8]{d*}.

As a consequence, we obtain the following combinatorial version of Theorem \ref{thm: MultLargeInt AVIP}:

\begin{thm}
	Let $K$ be a number field with ring of integers $\O_K$.
	Let $p(x) \in K[x]$ be an $\O_K$-valued intersective polynomial.
	Let $r,s \in \O_K$ be distinct and nonzero.
	Then for any set $E \subseteq \O_K$ with $d^*(E) > 0$ and any $\eps > 0$, the set
	\begin{equation*}
		\left\{ n \in \O_K :
		 d^* \left( E \cap (E - rp(n)) \cap (E - sp(n)) \right) > d^*(E)^3 - \eps \right\}
	\end{equation*}
	is $\text{AVIP}_{0,+}^*$ (in particular, it is syndetic).
	
	Moreover, if $\frac{s}{r} \in \Q$, then
	\begin{equation*}
		\left\{ n \in \O_K :
		 d^* \left( E \cap (E - rp(n)) \cap (E - sp(n)) \cap (E - (r+s)p(n)) \right) > d^*(E)^4 - \eps \right\}
	\end{equation*}
	is $\text{AVIP}_{0,+}^*$ (in particular, it is syndetic).
\end{thm}

As discussed above, the ergodicity assumption in Theorems \ref{thm: MultLargeInt} and \ref{thm: MultLargeInt AVIP}
precludes us from easily deducing finitary results along the lines of Corollary \ref{cor: finitary patterns}.
Nevertheless, we suspect that a finitary analogue holds, which we formulate below:

\begin{conj} \label{conj: finitary}
	Let $K$ be a degree $d$ number field with ring of integers $\O_K \cong \Z^d$.
	Suppose $p(x) \in K[x]$ is an $\O_K$-valued intersective polynomial.
	
	\begin{enumerate}[1.]
		\item	Let $r, s \in \O_K$ be distinct and nonzero.
			For any $\delta, \eps > 0$, there exists $N = N(\eps, \delta) \in \N$ such that:
			if $A \subseteq \{1, \dots, N\}^d$ with $|A| > \delta N^d$, then $A$ contains at least $(\delta^3 - \eps)N^d$
			configurations of the form $\{x, x + rp(n), x + sp(n)\}$ for some $n \ne 0$.
		\item	Let $r, s \in \O_K$ be distinct and nonzero such that $\frac{s}{r} \in \Q$
			(or more generally, no two conjugates of $\frac{s}{r}$ are negatives of each other).
			For any $\delta, \eps > 0$, there exists $N = N(\eps, \delta) \in \N$ such that:
			if $A \subseteq \{1, \dots, N\}^d$ with $|A| > \delta N^d$, then $A$ contains at least $(\delta^4 - \eps)N^d$
			configurations of the form $\{x, x + rp(n), x + sp(n), x + rp(n) + sp(n)\}$ for some $n \ne 0$.
	\end{enumerate}
\end{conj}

\noindent In the simplest case when $K = \Q$ and $p(n) = n$, Conjecture \ref{conj: finitary}
was posed as a question in \cite{bhk} and verified in \cite[Theorem 1.10]{green} and \cite[Theorem 1.12]{gt}.
For more general linear patterns (the case $p(n) = n$),
closely related finitary results were recently established in \cite{MIT, kovac}.

%%%%%%%%%%%%%%%%%%%%%%%%%%%%%%%%%%%%%%%%%%%%%%%%%%%%

\subsection{Outline of the paper}

The structure of the paper is as follows.
In Section \ref{sec: prelim}, we collect several useful facts that will be used repeatedly in the proofs of the main theorems.
Section \ref{sec: factors} is devoted to proving Theorems \ref{thm: rational Kronecker} and \ref{thm: nilfactor}
on characteristic factors corresponding to the polynomial configurations of interest via equidistribution results on nilmanifolds.
Using the knowledge of characteristic factors, we prove Khintchine-type results
(Theorems \ref{thm: IndLargeInt} and \ref{thm: MultLargeInt}) in Section \ref{sec: Khintchine}.
Finally, Section \ref{sec: refinements} handles the refinements of our Khintchine-type theorems
to conclude stronger combinatorial properties about the abundance of combinatorial configurations.

%%%%%%%%%%%%%%%%%%%%%%%%%%%%%%%%%%%%%%%%%%%%%%%%%%%%

\section{Preliminaries} \label{sec: prelim}

%%%%%%%%%%%%%%%%%%%%%%%%%%%%%%%%%%%%%%%%%%%%%%%%%%%%

\subsection{Rational Kronecker factor} \label{sec: rational Kronecker}

Recall that the Kronecker factor for an ergodic measure-preserving system is spanned by eigenfunctions.
As suggested by the name, the rational Kronecker factor will be spanned by eigenfunctions
with rational eigenvalues.
To make this precise, we need to define what it means for an eigenvalue (a group character) to be rational.

Since the additive group structure for the ring of integers in a degree $d$ extension of $\Q$ is $\Z^d$,
we say that a character $\chi : \Z^d \to \T$ is \emph{rational} if there is an element $(q_1, \dots, q_d) \in \Q^d$
such that
\begin{equation} \label{eq: rational character}
	\chi(n_1, \dots, n_d) = e^{2\pi i (q_1n_1 + \cdots + q_dn_d)}.
\end{equation}
For notational convenience, we will let $e : \R \to \T$ be the function $e(x) = e^{2\pi ix}$
so that we can write equation \eqref{eq: rational character} in the more compact form
\begin{equation*}
	\chi(n) = e(q \cdot n)
\end{equation*}
for the usual dot product $\cdot$ on $\R^d$.

The property of rational characters that we will utilize later on is periodicity.
Given a number field $K$ with ring of integers $\O_K$,
we say that a character $\chi : \O_K \to \T$ is \emph{periodic}, with \emph{period} $p \in \O_K$,
if for all $n, m \in \O_K$, we have
\begin{equation*}
	\chi(n + mp) = \chi(n).
\end{equation*}
To translate this back into language where rationality makes sense,
take an integral basis $\{b_1, \dots, b_d\}$ so that $(\O_K,+) \cong \bigoplus_{i=1}^d{\Z \cdot b_i}$.
Since $\hat{\Z^d} \cong \T^d$, there is an element $\alpha \in \T^d$ so that
\begin{equation*}
	\chi \left( \sum_{i=1}^d{n_ib_i} \right) = e(n \cdot \alpha)
\end{equation*}
for $n \in \Z^d$.
We can then say that $\chi : \O_K \to \T$ is \emph{rational} if $\alpha \in \Q^d/\Z^d$.
We now show that rationality and periodicity coincide:

\begin{lem} \label{lem: rational periodic}
	A character $\chi : \O_K \to \T$ is rational if and only if it is periodic.
\end{lem}

\begin{proof}
	Suppose $\chi$ is rational, say $\chi(n) = e(n \cdot q)$ with $q = (q_1, \dots, q_d) \in \Q^d$.
	Choose $D \in \Z$ so that $Dq_i \in \Z$ for every $i = 1, \dots, d$,
	and set $p := D \left( \sum_{i=1}^d{b_i} \right)$.
	We claim that $p$ is a period for $\chi$.
	Indeed, given $m = \sum_{i=1}^d{m_i b_i} \in \O_K$, we have
	\begin{equation*}
		mp = D \left( \sum_{i,j}{m_i b_i b_j} \right)
		 = \sum_k{\left( \sum_{i,j}{m_i c_{i,j,k}} \right) D b_k},
	\end{equation*}
	where $c_{i,j,k} \in \Z$ so that $b_ib_j = \sum_k{c_{i,j,k}b_k}$.
	Hence, $mp \cdot q = \sum_k{\left( \sum_{i,j}{m_i c_{i,j,k}} \right) D q_k} \in \Z$,
	so $\chi(mp) = e(mp \cdot q) = 1$ for every $m \in \O_K$. \\
	
	Conversely, suppose $\chi$ is periodic with period $p \in \O_K$.
	Let $\alpha \in \T^d$ such that $\chi(n) = e(n \cdot \alpha)$.
	Since $\{b_1, \dots, b_d\}$ is a $\Q$-basis for $K$, we can write $\frac{1}{p} = \sum_{i=1}^d{a_ib_i}$
	for some $a_1, \dots, a_d \in \Q$.
	Let $D \in \Z$ such that $Da_i \in \Z$ for every $i = 1, \dots, d$.
	Then $\frac{D}{p}$ is a $\Z$-linear combination of basis elements, so $\frac{D}{p} \in \O_K$.
	Now let $m_i = \frac{D}{p} b_i \in \O_K$.
	Since $\chi$ is $p$-periodic, we have
	\begin{equation*}
		1 = \chi(m_ip) = \chi(Db_i) = e(D\alpha_i).
	\end{equation*}
	That is, $D\alpha_i \in \Z$, so $\alpha_i \in \Q$ for all $i = 1, \dots, d$.
	Thus, $\chi$ is rational.
\end{proof}

Now we can give our definition:
\begin{defn}
	Let $K$ be a number field, and let $\O_K$ be its ring of integers.
	Let $\X = \left( X, \B, \mu, T \right)$ be an ergodic $\O_K$-system.
	The \emph{rational Kronecker factor} of $\X$, denoted by $\calK_{rat}(\X)$, is the factor
	generated by the algebra
	\begin{equation*}
		\overline{\spn{\left\{ f \in L^2(\mu) : T^nf = \chi(n)f~\text{for some rational character}~\chi : \O_K \to \T \right\}}}.
	\end{equation*}
\end{defn}

Building on Lemma \ref{lem: rational periodic}, we can characterize total ergodicity in several equivalent ways:

\begin{prop} \label{prop: tot erg}
	Let $K$ be a number field with ring of integers $\O_K$.
	Let $\X = (X, \B, \mu, T)$ be an ergodic $\O_K$-system.
	The following are equivalent:
	
	\begin{enumerate}[(i)]
		\item	The rational Kronecker factor $\calK_{rat}(\X)$ is trivial;
		\item	For every $r \in \O_K \setminus \{0\}$, $(T^{rn})_{n \in \O_K}$ is ergodic;
		\item	For every finite index subgroup $\Lambda \subseteq (\O_K, +)$,
			the action $(T^n)_{n \in \Lambda}$ is ergodic.
	\end{enumerate}
\end{prop}

\begin{proof}
	Since $r\O_K$ has finite index in $\O_K$, we trivially have the implication (iii)$\implies$(ii).
	We will show (i)$\implies$(iii) and (ii)$\implies$(i). \\
	
	(i)$\implies$(iii).
	Suppose $\calK_{rat}(\X)$ is trivial.
	Let $\Lambda \subseteq (\O_K,+)$ be a finite index subgroup, and suppose $T^nf = f$ for every $n \in \Lambda$.
	Then the orbit $\{T^nf : n \in \O_K\}$ is finite: it consists of the elements $T^mf$ for $m$ in a finite set $F$
	satisfying $\Lambda + F = \O_K$.
	In particular, the orbit is (pre-)compact, so $f$ is a linear combination of eigenfunctions,
	$f = \sum_i{c_if_i}$, with $T^nf_i = \chi_i(n)f_i$ for some characters $\chi_i : \O_K \to \T$.
	Since $T^nf = f$ for $n \in \Lambda$, we have $\chi_i(n) = 1$ for $n \in \Lambda$.
	Therefore, $\chi_i$ takes only the finitely many values $\chi_i(m)$, $m \in F$.
	It follows that $\chi_i$ is rational.
	But $\calK_{rat}(\X)$ is trivial, so in fact $\chi_i = 1$.
	Hence, $T^nf = f$ for every $n \in \O_K$.
	Since $(T^n)_{n \in \O_K}$ is ergodic, we have that $f$ is a constant function.
	Thus, $(T^n)_{n \in \Lambda}$ is ergodic. \\
	
	(ii)$\implies$(i).
	We prove the contrapositive.
	Suppose $\calK_{rat}(\X)$ is not trivial.
	Then there is a non-constant function $f \in L^2(\mu)$ and a rational character $\chi : \O_K \to \T$
	such that $T^nf = \chi(n)f$ for $n \in \O_K$.
	By Lemma \ref{lem: rational periodic}, $\chi$ is periodic, say with period $p$.
	That is, $\chi(n+pm) = \chi(n)$ for $n, m \in \O_K$.
	But then $T^{pn}f = \chi(pn)f = f$ for every $n \in \O_K$.
	Hence, $(T^{pn})_{n \in \N}$ is not ergodic, so (ii) fails.
\end{proof}

%%%%%%%%%%%%%%%%%%%%%%%%%%%%%%%%%%%%%%%%%%%%%%%%%%%%

\subsection{Nilsystems}

\begin{prop} \label{prop: tot erg conn}
	Let $K$ be a number field with ring of integers $\O_K$.
	Let $(X, \B, \mu, T)$ be an ergodic $\O_K$-nilsystem.
	Then $T$ is totally ergodic if and only if $X$ is connected.
\end{prop}

\begin{proof}
	Write $X = G/\Gamma$.
	Let $a : \O_K \to G$ be a homomorphism so that $T^nx = a(n) \cdot x$ for $n \in \O_K$ and $x \in X$.
	Let $x_0$ denote the image of the identity element in $X$. \\
	
	Suppose $T$ is totally ergodic, and let $X_0$ be the connected component of $x_0$.
	Since $X$ is compact, it is a disjoint union of finitely many translates of $X_0$,
	say $X = \bigcup_{i=0}^{k-1}{X_i}$ with $X_i = g_iX_0$.
	Hence, $G$ permutes the components $X_0, \dots, X_{k-1}$, giving a homomorphism $\varphi : G \to S_k$,
	where $S_k$ is the symmetric group on $k$ symbols.
	This in turn gives a homomorphism $\omega = \varphi \circ a : \O_K \to S_k$.
	Let $\Omega = \ker{\omega} \subseteq \O_K$.
	Since $S_k$ is a finite group, $\Omega$ has finite index in $\O_K$.
	Therefore, $(T^n)_{n \in \Omega}$ acts ergodically on $X$ (see Proposition \ref{prop: tot erg}(iii)).
	In particular, $\overline{\{a(n)x_0 : n \in \Omega\}} = X$.
	But for $n \in \Omega$, we have $a(n)X_i = X_i$, so $a(n)x_0 \in X_0$.
	Thus, $X = X_0$. \\
	
	Conversely, suppose $X$ is connected, and let $r \in \O_K \setminus \{0\}$.
	The group $r\O_K$ has finite index in $\O_K$, so let $s_0, \dots, s_{k-1} \in \O_K$
	such that $\bigcup_{i=0}^{k-1}{(r\O_K + s_i)} = \O_K$.
	Let $Y := \overline{\{a(rn)x_0 : n \in \O_K\}}$.
	Then by ergodicity of $T$, we have $X = \bigcup_{i=0}^{k-1}{a(s_i)Y}$.
	We claim that for $0 \le i, j \le k-1$, the sets $a(s_i)Y$ and $a(s_j)Y$ are either disjoint or identical.
	Indeed, suppose $x \in a(s_i)Y \cap a(s_j)Y$.
	Then there are sequences $(n_t)_{t \in \N}$ and $(m_t)_{t \in \N}$ in $\O_K$ such that
	\begin{equation*}
		a(s_i) \lim_{t \to \infty}{a(rn_t)x_0} = a(s_j) \lim_{t \to \infty}{a(rm_t)x_0} = x.
	\end{equation*}
	Let $(\gamma_t)_{t \in \N}$ and $(\delta_t)_{t \in \N}$ be sequences in $\Gamma$ and $g \in G$ with $g\Gamma = x$
	so that
	\begin{equation*}
		a(s_i) \lim_{t \to \infty}{a(rn_t) \gamma_t} = a(s_j) \lim_{t \to \infty}{a(rm_t) \delta_t} = g.
	\end{equation*}
	Then
	\begin{align*}
		a(s_j - s_i) \lim_{t \to \infty}{a \left( r (m_t - n_t) \right) x_0}
		 & = \lim_{t \to \infty}{\gamma_t \left( a(s_i) a(rn_t) \gamma_t \right)^{-1} a(s_j) a(rm_t) \delta_t \Gamma} \\
		 & = \lim_{t \to \infty}{\gamma_t g^{-1} g \Gamma} = x_0.
	\end{align*}
	It follows that $a(s_j-s_i)Y = Y$, so multiplying by $a(s_i)$, we get $a(s_i)Y = a(s_j)Y$.
	
	But then we have written $X$ as a finite disjoint union of closed sets.
	Since $X$ is connected, we must have $a(s_i)Y = X$ for every $i = 0, \dots, k-1$.
	In particular, $Y = X$, so $(T^{rn})_{n \in \O_K}$ is ergodic.
\end{proof}

%%%%%%%%%%%%%%%%%%%%%%%%%%%%%%%%%%%%%%%%%%%%%%%%%%%%

\subsection{Weyl systems}

The results in this paper depend critically on understanding polynomial orbits in \emph{Weyl systems}.
Following \cite{bll-weyl}, we call a topological dynamical system $(X, T)$ a \emph{Weyl system} if
$X$ is a compact abelian Lie group and $T$ is a $\Z^d$-action by unipotent affine tranformations.
In proving our multiple recurrence results, we will focus our attention on \emph{connected Weyl systems},
that is Weyl systems where $X$ is connected (and hence a torus).

The main result on polynomial orbits is the following:

\begin{prop}[cf. \cite{bll-weyl}, Proposition 3.2] \label{prop: finite subtori}
	Let $(X, T)$ be a $\Z^d$-Weyl system and $p_1, \dots, p_m : \Z^d \to \Z^d$ polynomials.
	Then for every $x \in X$,
	$Y := \overline{\left\{ \left( T^{p_1(n)}x, \dots, T^{p_m(n)}x \right) : n \in \Z^d \right\}}$
	is a union of finitely many subtori $(Y_w)_{w \in W}$ of $X^m$.
	Moreover, there is a homomorphism $\omega : \Z^d \to W$ such that the sequence
	$\left( T^{p_1(n)}x, \dots, T^{p_m(n)}x \right)_{n \in \omega^{-1}(w)}$
	is well-distributed in $Y_w$ for each $w \in W$.
\end{prop}

\noindent This can be seen via a multivariable version of Weyl's theorem on polynomial equidistribution in tori
(see the explanation of \cite[Proposition 3.2]{bll-weyl})
or as a special case of a more general result due to Leibman:

\begin{thm}[\cite{leib-poly_nil2}, Theorem B*] \label{thm: nil orbits}
	Let $X = G/\Gamma$ be a nilmanifold.
	Let $g : \Z^d \to G$ be a polynomial map, and let $x \in X$.
	There is a connected closed subgroup $H \subseteq G$, a homomorphism $\omega : \Z^d \to W$
	onto a finite group $W$, and a set $\{x_w : w \in W\} \subseteq X$ such that the sets
	$Y_w := Hx_w$, $w \in W$, are closed in $X$ and $\left( g(n)x \right)_{n \in \omega^{-1}(w)}$
	is well-distributed in $Y_w$ for every $w \in W$.
\end{thm}

As a consequence, we can deduce a simple criterion for checking
that a polynomial sequence is well-distributed in a torus.
First we need some notation.
For a sequence $u : \Z^d \to \T^m$ with polynomial coordinates $u(n) = \left( u_1(n), \dots, u_m(n) \right)$, we write
\begin{equation*}
	\spn(u) := \spn_{\R}{\left\{ \left( u_1(x), \dots, u_m(x) \right) : x \in \R^d \right\}}.
\end{equation*}

\begin{cor} \label{cor: span}
	Let $\alpha_1, \dots, \alpha_r$ be rationally independent irrational elements of $\T$.
	Let $u_1, \dots, u_r : \Z^l \to \Z^m$ be polynomials with zero constant term.
	Then the sequence
	\begin{equation*}
		\left( u_1(n)\alpha_1 + \cdots + u_r(n)\alpha_r \right)_{n \in \Z^l}
	\end{equation*}
	is well-distributed in the subtorus $\spn(u_1) + \cdots + \spn(u_r) \pmod{1}$ of $\T^m$.
\end{cor}

\begin{proof}
	The case $d = 1$ is handled by \cite[Corollary 3.3]{bll-weyl}.
	The same proof works for general $d \in \N$.
\end{proof}

%%%%%%%%%%%%%%%%%%%%%%%%%%%%%%%%%%%%%%%%%%%%%%%%%%%%

\subsection{Properties of polynomials}

\begin{defn}
	The polynomials $p_1, \dots, p_m \in \Q[x_1, \dots, x_d]$ are \emph{algebraically independent} (over $\Q$)
	if, for every nonzero $f \in \Q[x_1, \dots, x_m]$, the polynomial $f(p_1, \dots, p_m)$ is nonzero.
\end{defn}

\begin{prop} \label{prop: alg ind coord}
	Let $K$ be a number field with ring of integers $\O_K$.
	Let $p(x) \in K[x]$ be a nonconstant $\O_K$-valued polynomial.
	Fix an integral basis $\{b_1, \dots, b_d\} \subseteq \O_K$,
	and let $p_1, \dots, p_d \in \Q[x_1, \dots, x_d]$ be $\Z$-valued polynomials so that
	\begin{equation*}
		p \left( \sum_{i=1}^d{n_ib_i} \right) = \sum_{i=1}^d{p_i(n_1, \dots, n_d)b_i}.
	\end{equation*}
	Then the polynomials $p_1, \dots, p_d$ are algebraically independent (over $\Q$).
\end{prop}

\begin{proof}
	By \cite[Chapter I, 11.4]{lefschetz}, it suffices to check that the Jacobian matrix
	\begin{equation*}
		J := \left( \begin{array}{ccc}
			\frac{\partial p_1}{\partial x_1} & \cdots & \frac{\partial p_d}{\partial x_1} \\
			\vdots & \ddots & \vdots \\
			\frac{\partial p_1}{\partial x_d} & \cdots & \frac{\partial p_d}{\partial x_d}
		\end{array} \right)
	\end{equation*}
	has full rank.
	
	But the $i$th row of the Jacobian matrix is given by
	\begin{equation*}
		\frac{\partial p}{\partial x_i}(x) = \lim_{h \to 0}{\frac{p(x + hb_i) - p(x)}{h}}
		 = \left( \lim_{h \to 0}{\frac{p(x + hb_i) - p(x)}{hb_i}} \right) b_i = p'(x) b_i.
	\end{equation*}
	Since $p$ is nonconstant, $p'(x) \not\equiv 0$.
	Moreover, $\{b_1, \dots, b_d\}$ is linearly independent over $\Q$, so the rows of $J$ are linearly independent.
	Therefore, $J$ has full rank, so $p_1, \dots p_d$ are algebraically independent.
\end{proof}

\begin{prop} \label{prop: ind coord}
	Let $K$ be a number field with ring of integers $\O_K$, and let $\{b_1, \dots, b_d\}$ be an integral basis in $\O_K$.
	Let $\{p_1, \dots, p_k\} \subseteq K[x]$ be an independent family of polynomials (over $K$).
	For each $i = 1, \dots, k$, let $p_{i,1}, \dots, p_{i,d} \in \Q[x_1, \dots, x_d]$ be the coordinate polynomials so that
	\begin{equation*}
		p_i \left( \sum_{j=1}^d{x_jb_j} \right) = \sum_{j=1}^d{p_{i,j}(x_1, \dots, x_d)b_j}.
	\end{equation*}
	Then the family $\{p_{i,j} : 1 \le i \le k, 1 \le j \le d\} \subseteq \Q[x_1, \dots, x_d]$
	is independent (over $\Q$).
	That is, for any $(c_{i,j})_{1 \le i \le k, 1 \le j \le d} \in \Q^{kd} \setminus \{0\}$, the polynomial
	$\sum_{i=1}^k{\sum_{j=1}^d{c_{i,j}p_{i,j}}}$ is nonconstant.
\end{prop}

\begin{proof}
	First, since $\{b_1, \dots, b_d\}$ is linearly independent over $\Q$,
	the family $\{b_jp_i : 1 \le i \le k, 1 \le j \le d\} \subseteq K[x]$ is independent over $\Q$.
	
	Now let $q_{i,j} \in \Q[x_1, \dots, x_d]$ be the $b_1$-coordinate of $b_jp_i$.
	That is,
	\begin{equation*}
		b_jp_i \left( \sum_{l=1}^d{x_lb_l} \right) = q_{i,j}(x_1, \dots, x_d)b_1 + r_{i,j}(x_1, \dots, x_d),
	\end{equation*}
	where $r_{i,j}(x_1, \dots, x_d) \in \spn_{\Q}\{b_2, \dots, b_d\}$.
	We claim that $\{q_{i,j} : 1 \le i \le k, 1 \le j \le d\}$ is independent over $\Q$.
	Suppose not.
	Then for some $(c_{i,j})_{1 \le i \le k, 1 \le j \le d} \in \Q^{kd} \setminus \{0\}$ and some $c \in \Q$, we have
	\begin{equation*}
		\sum_{i=1}^k{\sum_{j=1}^d{c_{i,j}q_{i,j}(x_1, \dots, x_d)}} = c.
	\end{equation*}
	Then
	\begin{equation*}
		Q \left( \sum_{l=1}^d{x_lb_l} \right)
		 := \sum_{i=1}^k{\sum_{j=1}^d{c_{i,j}b_jp_i \left( \sum_{l=1}^d{x_lb_l} \right)}}
		 = cb_1 + \sum_{i=1}^k{\sum_{j=1}^d{c_{i,j}r_{i,j}(x_1, \dots, x_d)}}.
	\end{equation*}
	Hence, for the polynomial function $f \left( \sum_{l=1}^d{x_lb_l} \right) := x_1 - c$, we have $f(Q) = 0$.
	By Proposition \ref{prop: alg ind coord}, it follows that $Q$ is constant.
	But $\{b_jp_i : 1 \le i \le k, 1 \le j \le d\}$ is independent over $\Q$, so this is a contradiction.
	
	For $1 \le j, l, m \le d$, let $a_{j,l,m} \in \Z$ so that $b_jb_l = \sum_{m=1}^d{a_{j,l,m}b_m}$.
	By direct computation, we have
	\begin{equation*}
		b_jp_i \left( \sum_{l=1}^d{x_lb_l} \right)
		 = \sum_{l=1}^d{p_{i,l}(x_1, \dots, x_d) \sum_{m=1}^d{a_{j,l,m}b_m}}
		 = \sum_{m=1}^d{ \left( \sum_{l=1}^d{a_{j,l,m}p_{i,l}(x_1, \dots, x_d)} \right) b_m}.
	\end{equation*}
	Thus,
	\begin{equation*}
		q_{i,j}(x_1, \dots, x_d) = \sum_{l=1}^d{a_{j,l,1}p_{i.l}(x_1, \dots, x_d)}
		 \in \spn_{\Q}{\{p_{i,1}, \dots, p_{i,d}\}}.
	\end{equation*}
	Therefore, $\spn{\left( \{1\} \cup \{q_{i,j} : 1 \le i \le k, 1 \le j \le d\} \right)}
	 \subseteq \spn{\left( \{1\} \cup \{p_{i,j} : 1 \le i \le k, 1 \le j \le d\} \right)}$.
	It follows that $\{p_{i,j} : 1 \le i \le k, 1 \le j \le d\}$ is independent over $\Q$.
\end{proof}

\begin{lem} \label{lem: intersective}
	Let $K$ be a number field with ring of integers $\O_K$.
	Let $\{p_1, \dots, p_k\} \subseteq K[x]$ be jointly intersective $\O_K$-valued polynomials.
	Let $r \in \O_K \setminus \{0\}$.
	Then there exists $\xi \in \O_K$ and $D \in \O_K \setminus \{0\}$ such that
	\begin{equation*}
		p_i \left( \xi + D\O_K \right) \subseteq r\O_K
	\end{equation*}
	for $i = 1, \dots, k$.
\end{lem}

\begin{proof}
	The subgroup $r\O_K$ has finite index in $\O_K$,
	so there exists $\xi \in \O_K$ such that $p_i(\xi) \in r\O_K$ for $i = 1, \dots, k$.
	
	Fix $1 \le i \le k$.
	Now write $p_i(x) = a_m x^m + \cdots + a_1 x + a_0$ with $a_0, a_1, \dots, a_m \in K$.
	Since $p_i$ is $\O_K$-valued, we have $a_0 = p_i(0) \in \O_K$.
	Let $D_i \in \O_K$ so that $D_ia_j \in \O_K$ for all $j = 1, \dots, m$.
	We claim $p_i(\xi + D_ir\O_K) \subseteq r\O_K$.
	Indeed, for $n \in \O_K$, we have
	\begin{align*}
		p \left( \xi + D_irn \right) & = p_i(\xi) + \sum_{j=1}^m{\sum_{l=1}^j{a_j\binom{j}{l}(D_irn)^l\xi^{j-l}}} \\
		 & = p_i(\xi) + r \cdot \sum_{j=1}^m{\left( D_ia_j \sum_{l=1}^j{\binom{j}{l}D^{l-1}k^{l-1}n^l \xi^{l-j}} \right)}
		 \in r \O_K.
	\end{align*}
	
	Taking $D = r \cdot \text{lcm}(D_1, \dots, D_k)$ completes the proof.
\end{proof}

%%%%%%%%%%%%%%%%%%%%%%%%%%%%%%%%%%%%%%%%%%%%%%%%%%%%

\subsection{Eligible collections}

Theorem \ref{thm: rational Kronecker} and Theorem \ref{thm: nilfactor} each establish characteristic factors
for certain polynomial configurations in ergodic systems.
In both cases, it is significantly easier to deal with totally ergodic systems.
The notion of \emph{eligible} collections, introduced by Frantzikinakis in \cite{fra} for $\Z$-valued polynomials,
can be utilized to reduce the ergodic case to the simpler case in which the system is totally ergodic.

\begin{defn} \label{defn: eligible}
	Let $\P$ be a collection of families of $k$ $\O_K$-valued polynomials.
	We say that $\P$ is \emph{eligible} if for any $\{p_1, \dots, p_k\} \in \P$, we have
	\begin{enumerate}[(i)]
		\item	$\{p_1(n) - p_1(0), \dots, p_k(n) - p_k(0)\} \in \P$;
		\item	$\{p_1(rn+s), \dots, p_k(rn+s)\} \in \P$ for any $r, s \in \O_K$ with $r \ne 0$;
		\item	$\{cp_1(n), \dots, cp_k(n)\} \in \P$ for any $c \in K \setminus \{0\}$
			such that $cp_i$ is $\O_K$-valued for $i = 1, \dots, k$.
	\end{enumerate}
\end{defn}

\begin{prop}[cf. \cite{fra}, Proposition 4.1] \label{prop: eligible}
	Let $\P$ be eligible.
	Suppose that for some $m \in \N$, the nilfactor $\calZ_m$ is characteristic
	for every $P \in \P$ in totally ergodic systems.
	Then $\calZ_m$ is characteristic for every $P \in \P$ in ergodic systems.
\end{prop}

\begin{proof}
	Let $(X, \B, \mu, T)$ be an ergodic $\O_K$-system.
	Let $f_1, \dots, f_k \in L^{\infty}(\mu)$, and suppose $\E{f_i}{\calZ_m} = 0$ for some $i = 1, \dots, k$.
	Without loss of generality, $i = 1$.
	We want to show
	\begin{equation*}
		\UClim_{n \in \O_K}{ \prod_{i=1}^k{T^{p_i(n)}f_i} } = 0
	\end{equation*}
	in $L^2(\mu)$.
	Shifting by the constant terms and using property (i), we may assume $p_i(0) = 0$.
	
	By Theorem \ref{thm: polynomial nil}, we may assume that $X = G/\Gamma$ is a nilmanifold
	and $T$ acts by niltranslations $T^nx = a(n)x$ with $a(n) \in G$.
	We claim that there exists $r \in \O_K$ such that the (finitely many) ergodic components
	of the action $\left( T^{rn} \right)_{n \in \O_K}$ are totally ergodic.
	Let $\calZ$ be the Kronecker factor of $(X,T)$.
	This is an action by rotations on an abelian Lie group of the form
	$\Z_{a_1} \times \cdots \times \Z_{a_d} \times \T^{c}$ with $a_1, \dots, a_d \in \N, c \in \N \cup \{0\}$.
	Set $a := \prod_{j=1}^d{a_j} \in \N$.
	Letting $r = a(b_1 + \cdots + b_d)$, where $\{b_1, \dots, b_d\}$ is an integral basis for $\O_K$,
	we then have that the ergodic components of $\left( T^{rn} \right)_{n \in \O_K}$ are totally ergodic
	(we have trivialized the rational component of the Kronecker factor).
	
	Since $p_i(0) = 0$ for each $i = 1, \dots, k$, there exist $D \in \O_K \setminus \{0\}$
	so that the polynomials $q_i(n) := r^{-1}p_i(Dn)$ are $\O_K$-valued by Lemma \ref{lem: intersective}.
	By properties (ii) and (iii), $\{q_1, \dots, q_k\} \in \P$.
	
	Now, $\left( T^{rn} \right)_{n \in \O_K}$ has finitely many ergodic components
	and $\calZ_m(T^{rn}) \subseteq \calZ_m(T^n)$, so
	\begin{equation*}
		\E{f_1}{\calZ^{(j)}_m} = 0,
	\end{equation*}
	where $\calZ^{(j)}_m$ is the nilfactor for the $j$th ergodic component of $T^{rn}$.
	Summing over the finitely many ergodic components of $T^{rn}$, we thus have
	\begin{equation*}
		\UClim_{n \in \O_K}{ \prod_{i=1}^k{T^{p_i(Dn)}f_i} }
		 = \UClim_{n \in \O_K}{ \prod_{i=1}^k{T^{r \cdot q_i(n)}f_i} } = 0.
	\end{equation*}
	
	Note that by the proof of Lemma \ref{lem: intersective}, $p_i(Dn + s) \equiv p_i(s) \pmod{r\O_K}$.
	Hence, $q^{(s)}_i(n) := r^{-1}(p_i(Dn+s) - p_i(s))$ is $\O_K$-valued.
	Moreover, since $\P$ is eligible, we have $\left\{ q^{(s)}_1, \dots, q^{(s)}_k \right\} \in \P$.
	By assumption, $\E{f_i}{\calZ_m} = 0$ for some $i = 1, \dots, k$.
	It follows that $\E{T^{p_i(s)}f_i}{\calZ_m} = 0$, since $\calZ_m$ is $T$-invariant.
	Thus, by the argument in the previous paragraph, we have
	\begin{align*}
		\UClim_{n \in \O_K}{ \prod_{i=1}^k{T^{p_i(Dn+s)}f_i} }
		 & = \UClim_{n \in \O_K}{ \prod_{i=1}^k{T^{p_i(Dn+s)-p_i(s)}(T^{p_i(s)}f_i)} } \\
		 & = \UClim_{n \in \O_K}{ \prod_{i=1}^k{T^{r \cdot q^{(s)}_i(n)}(T^{p_i(s)}f_i)} } = 0
	\end{align*}
	for $s \in \O_K/r\O_K$.
	This completes the proof.
\end{proof}

%%%%%%%%%%%%%%%%%%%%%%%%%%%%%%%%%%%%%%%%%%%%%%%%%%%%
%%%%%%%%%%%%%%%%%%%%%%%%%%%%%%%%%%%%%%%%%%%%%%%%%%%%

\section{Characteristic factors} \label{sec: factors}

%%%%%%%%%%%%%%%%%%%%%%%%%%%%%%%%%%%%%%%%%%%%%%%%%%%%

\subsection{Proof of Theorem \ref{thm: rational Kronecker}}

We want to prove that the rational Kronecker factor, $\calK_{rat}$, is characteristic for the average
\begin{equation*}
	\UClim_{n \in \O_K}{\prod_{i=1}^k{T^{p_i(n)}f_i}}
\end{equation*}
when $p_1, \dots, p_k \in K[x]$ are independent $\O_K$-valued polynomials.

We will first prove a special case:

\begin{thm} \label{thm: tot erg}
	Let $K$ be a number field with ring of integers $\O_K$.
	Suppose $p_1, \dots, p_k \in K[x]$ are independent $\O_K$-valued polynomials.
	If $\left( X, \B, \mu, T \right)$ is a totally ergodic $\O_K$-system
	and $f_1, \dots, f_k \in L^{\infty}(\mu)$, then
	\begin{equation*}
		\UClim_{n \in \O_K}{\prod_{i=1}^k{T^{p_i(n)}f_i}}
		 = \prod_{i=1}^k{\int_X{f_i~d\mu}}.
	\end{equation*}
\end{thm}

\begin{rem}
	After a previous version of this paper appeared on \emph{arXiv},
	Best and Ferr\'{e} Moragues reproved Theorem \ref{thm: tot erg}
	using a different method (see \cite[Thoerem 1.6]{bf}).
\end{rem}

%%%%%%%%%%%%%%%%%%%%%%%%%%%%%%%%%%%%%%%%%%%%%%%%%%%%

\subsubsection{Reduction to Weyl systems}

Recall that a sequence $(x_n)_{n \in \O_K}$ in a compact topological space $X$ is \emph{well-distributed}
with respect to a probability measure $\mu$ on $X$ if $\UClim_{n \in \O_K}{\delta_{x_n}} = \mu$ in the weak-* topology.
That is, for any continuous function $f : X \to \C$ and any F{\o}lner sequence $(\Phi_N)_{N \in \N}$ in $(\O_K, +)$,
one has
\begin{equation*}
	\frac{1}{|\Phi_N|} \sum_{n \in \Phi_N}{f(x_n)} \to \int_X{f~d\mu}.
\end{equation*}
By Theorem \ref{thm: polynomial nil}, Theorem \ref{thm: tot erg}
is equivalent to the following equidistribution result:

\begin{thm} \label{thm: top tot erg}
	Let $K$ be a number field with ring of integers $\O_K$.
	Let $(X, \B, \mu, T)$ be a totally ergodic $\O_K$-nilsystem.
	Let $\{p_1, \dots, p_k\} \subseteq K[x]$ be independent $\O_K$-valued polynomials.
	Then for almost every $x \in X$, the sequence $\left( T^{p_1(n)}x, \dots, T^{p_k(n)}x \right)_{n \in \O_K}$
	is well-distributed in $X^k$.
\end{thm}

Having reduced to an equidistribution result on nilmanifolds, we can now make several more reductions.
First, by Proposition \ref{prop: tot erg conn}, the nilmanifold in Theorem \ref{thm: top tot erg} is necessarily connected,
since it admits a totally ergodic action by niltranslations.
Next, by Proposition \ref{prop: ind coord}, we may expand the polynomials $p_1, \dots, p_k$
in coordinates with respect to an integral basis in order to obtain an independent family of $\Z$-valued polynomials.
Hence, Theorem \ref{thm: top tot erg} follows from:

\begin{thm} \label{thm: top tot erg Z^l}
	Let $d, k, l \in \N$.
	Let $\left( X, \B, \mu, T \right)$ be an ergodic, connected $\Z^l$-nilsystem.
	Let $\{p_{i,j} : 1 \le i \le k, 1 \le j \le l\} \subseteq \Q[x_1, \dots, x_d]$ be a family of independent
	$\Z$-valued polynomials.
	Then the sequence
	\begin{equation*}
		\left( \prod_{j=1}^l{T_j^{p_{1,j}(n)}}x, \dots, \prod_{j=1}^l{T_j^{p_{k,j}(n)}}x \right)_{n \in \Z^d}
	\end{equation*}
	is well-distributed in $X^k$ for every $x$ in a co-meager set of full measure.
\end{thm}

Now we will reduce from a connected nilystem to the case that $(X,T)$ is a Weyl system,
i.e. $X$ is a finite-dimensional torus and $T$ acts by unipotent affine transformations.
Let $G_0$ be the connected component of the identity in $G$, $Z = X/[G_0, G_0]$,
and $\pi : G \to Z$ the projection map.
The following result of Leibman shows that we can reduce to studying orbits in $Z$:

\begin{thm}[\cite{leib-poly_nil2}, Theorem C] \label{thm: proj}
	Let $X = G/\Gamma$ be a connected nilmanifold, $x \in X$, and $g : \Z^d \to G$ a polynomial map.
	The following are equivalent:
	\begin{enumerate}[(i)]
		\item	the orbit $\left\{ g(n)x : n \in \Z^d \right\}$ is dense in $X$;
		\item	$\left\{ g(n)\pi(x) : n \in \Z^d \right\}$ is dense in $Z$;
		\item	$\left( g(n)x \right)_{n \in \Z^d}$ is well-distributed in $X$;
		\item	$\left( g(n)\pi(x) \right)_{n \in \Z^d}$ is well-distributed in $Z$.
	\end{enumerate}
\end{thm}

\begin{lem}[cf. \cite{frakra1}, Proposition 2.1]
	Without loss of generality, $G_0$ is abelian.
\end{lem}

\begin{proof}
	Use Theorem \ref{thm: proj} to reduce to the projection onto $Z$.
	Now, the group $(G/[G_0,G_0])_0$ is abelian, and a factor of a totally ergodic system is totally ergodic.
\end{proof}

\begin{lem}[cf. \cite{frakra1}, Propositions 3.1 and 3.2] \label{lem: affine}
	Without loss of generality, $(X,T)$ is a connected Weyl system.
\end{lem}

\begin{proof}
	To reduce from a connected nilsystem such that $G_0$ is abelian to a connected Weyl system,
	see \cite[Proposition 3.1]{frakra1}.
	The isomorphism between a niltranslation and a unipotent affine transformation
	does not depend on the element of $G$ defining the niltranslation,
	so the result still holds for $d$ commuting niltranslations.
\end{proof}

We have therefore reduced Theorem \ref{thm: tot erg} to the following result about
well-distribution of polynomial orbits for unipotent affine actions on tori:

\begin{thm} \label{thm: Z^l equidistribution}
	Let $d, l, k, m \in \N$.
	Let $\{p_{i,j} : 1 \le i \le k, 1 \le j \le l\} \subseteq \Q[x_1, \dots, x_d]$ be $\Z$-valued and independent over $\Q$.
	Let $T_1, \dots, T_l : \T^m \to \T^m$ be commuting unipotent affine transformations
	generating an ergodic $\Z^l$-action.
	Then the polynomial sequence
	\begin{equation*}
		\left( \prod_{j=1}^l{T_j^{p_{1,j}(n)}}x, \dots, \prod_{j=1}^l{T_j^{p_{k,j}(n)}}x \right)_{n \in \Z^d}
	\end{equation*}
	is well-distributed in $\T^{mk}$ for all $x$ in a co-meager set of full measure.
\end{thm}

%%%%%%%%%%%%%%%%%%%%%%%%%%%%%%%%%%%%%%%%%%%%%%%%%%%%

\subsubsection{Equidistribution of $\Z^l$-polynomial sequences}

In order to prove Theorem \ref{thm: Z^l equidistribution}, we will use two classic results in equidistribution.
The first is a multivariable version of Weyl's polynomial equidistribution theorem.

\begin{lem}[cf. \cite{weyl}, Satz 20] \label{lem: Weyl}
	Fix $d \in \N$.
	Let $p \in \R[x_1, \dots, x_d]$.
	If at least one coefficient of $p$ other than the constant term is irrational,
	then the $\Z^d$-sequence $\left( p(n_1, \dots, n_d) \right)_{n \in \Z^d}$ is well-distributed mod 1.
\end{lem}

\begin{rem}
	Weyl proved Lemma \ref{lem: Weyl} in the case $d=2$, with indications of how to prove the case of general $d$,
	for F{\o}lner sequences that are increasing dilations of a fixed set.
	Lemma \ref{lem: Weyl} in its full generality is today an easy exercise
	with the help of an appropriate variant of the van der Corput trick (see, e.g. \cite[Lemma A6]{bm}).
\end{rem}

The next lemma allows one to reduce equidistribution in a multidimensional torus to equidistribution in the circle.

\begin{lem} \label{lem: ud linear comb}
	Fix $d, m \in \N$.
	A $\Z^d$-sequence $u : \Z^d \to \T^m$ is well-distributed in $\T^m$ if and only if
	for every $c \in \Z^m \setminus \{0\}$, the sequence $c \cdot u(n) = c_1u_1(n) + \cdots + c_mu_m(n)$
	is well-distributed in $\T$.
\end{lem}
\begin{proof}
	See \cite[Theorem 6.3]{kn} for the case $d = 1$.
	The same argument works for general $d \in \N$.
\end{proof}

With these two lemmas at hand, we are now ready to prove Theorem \ref{thm: Z^l equidistribution}.

\begin{proof}[Proof of Theorem \ref{thm: Z^l equidistribution}]
	For each $j = 1, \dots, l$, we can write $T_jx = A_jx + \alpha_j$ for some unipotent $(m \times m)$-matrix $A_j$ with integer entries
	and a vector $\alpha_j \in \T^m$.
	Since the matrices $A_1, \dots, A_l$ commute, they are simultaneously triangularizable.
	That is, there is a matrix $P$ with rational entries and lower-triangular matrices $B_j$ such that $A_jP = PB_j$.
	Multiplying $P$ by a common denominator of its entries, we may assume that $P$ has integer entries.
	Then $P$ is well-defined as a surjective endomorphism of $\T^m$.
	(One can show that in general, $P$ cannot be assumed to be an automorphism.)
	Let $\beta_j \in \T^m$ such that $P\beta_j = \alpha_j$, and set $S_jx := B_jx + \beta_j$.
	Then we have $T_jPx = A_jPx + \alpha_j = PB_jx + P\beta_j = PS_jx$.
	That is, $T$ is a factor of $S$ with factor map $P : \T^m \to \T^m$.
	
	Now we check that $S$ is an ergodic $\Z^l$-action on $\T^m$.
	For $n = (n_1, \dots, n_l) \in \Z^l$, let $S^n := \prod_{j=1}^l{S_j^{n_j}}$ and $T^n := \prod_{j=1}^l{T_j^{n_j}}$.
	Suppose $A \subseteq \T^m$ is $S$-invariant.
	That is, $S^nA = A$ for every $n \in \Z^l$.
	Applying the factor map $P : \T^m \to \T^m$, we have $PA = PS^nA = T^nPA$, so $PA$ is a $T$-invariant set.
	But $T$ is ergodic by assumption, so $\mu(PA) \in \{0,1\}$.
	If $\mu(PA) = 0$, then $\mu(A) \le \mu(P^{-1}PA) = \mu(A) = 0$.
	On the other hand, if $\mu(PA) = 1$, then $\mu(A) \ge \frac{1}{|\det(P)|}$.
	
	Assume $A$ is an $S$-invariant set of minimal positive measure so that $\left. S \right|_A : A \to A$ is ergodic.
	Let $x \in A$ be a generic point for $\left. S \right|_A$.
	Then by Proposition \ref{prop: finite subtori}, $A$ differs from the set
	$Y = \overline{\left\{S^nx : n \in \Z^l\right\}}$ by a null-set, and $Y$ is a subtorus of $\T^m$.
	But $\mu(Y) = \mu(A) > 0$, so $Y = \T^m$.
	Thus, for any $S$-invariant set $A$ of positive measure, we have $\mu(A) = 1$.
	Therefore, $S$ is ergodic.
	
	The above argument shows that, without loss of generality, we may assume
	that the transformations $T_j$ are of the form
	\begin{equation*}
		T_jx = \left( x_1 + \alpha^{(j)}_1, x_2 + a^{(j)}_{2,1}x_1 + \alpha^{(j)}_2, \dots,
		 x_m + \sum_{r=1}^{m-1}{a^{(j)}_{m,r}x_r} + \alpha^{(j)}_m \right)
	\end{equation*}
	for $x = (x_1, \dots, x_m) \in \T^m$.
	
	Let $u_x(n) := \left( T^{p_1(n)}x, \dots, T^{p_k(n)}x \right)$
	for $x \in \T^m$ and $n \in \Z^d$, where $T^{p_i(n)}$ denotes the transformation $\prod_{j=1}^l{T_j^{p_{i,j}(n)}}$.
	We now break the proof into two cases, depending on the coefficients $a^{(j)}_{s,r}$. \\
	
	First consider the case $a^{(j)}_{s,r} = 0$ for all $1 \le j \le l$, $2 \le s \le m$, and $1 \le r \le s-1$.
	That is,
	\begin{equation*}
		T_jx = \left( x_1 + \alpha^{(j)}_1, x_2 + \alpha^{(j)}_2, \dots, x_m + \alpha^{(j)}_m \right)
		 = x + \alpha^{(j)}
	\end{equation*}
	is a toral rotation for $j = 1, \dots, l$.
	Let $\alpha : \Z^l \to \T^m$ be the homomorphism $\alpha(n_1, \dots, n_l) = \sum_{j=1}^l{n_j\alpha^{(j)}}$.
	Then for any $x \in \T^m$, we have
	\begin{equation*}
		u_x(n) = x + \left( \alpha(p_1(n)), \dots, \alpha(p_k(n)) \right) = x + u_0(n).
	\end{equation*}
	
	Now let $(c_{i,s})_{1 \le i \le k, 1 \le s \le m} \in \Z^{mk} \setminus \{0\}$.
	By Lemma \ref{lem: ud linear comb}, it suffices to show
	\begin{equation*}
		c \cdot u_0(n) = \sum_{i=1}^k{c_i \cdot \alpha(p_i(n))}
	\end{equation*}
	is well-distributed in $\T$.
	For each $i = 1, \dots, k$ and $n \in \Z^l$, we have
	\begin{equation*}
		c_i \cdot \alpha(n) = \sum_{s=1}^m{c_{i,s} \sum_{j=1}^l{n_j\alpha^{(j)}_s}}
		 = \sum_{j=1}^l{n_j \sum_{s=1}^m{c_{i,s}\alpha^{(j)}_s}}.
	\end{equation*}
	Letting
	$\beta_i = \left( \sum_{s=1}^m{c_{i,s}\alpha^{(1)}_s}, \dots, \sum_{s=1}^m{c_{i,s}\alpha^{(l)}_s} \right) \in \T^l$,
	we therefore have
	\begin{equation*}
		c_i \cdot \alpha(n) = \beta_i \cdot n.
	\end{equation*}
	Thus,
	\begin{equation*}
		c \cdot u_0(n) = \sum_{i=1}^k{\beta_i \cdot p_i(n)}
		 = \sum_{i=1}^k{\sum_{j=1}^l{\beta_{i,j}p_{i,j}(n)}},
	\end{equation*}
	which is well-distributed by Lemma \ref{lem: Weyl}. \\
	
	Now suppose $a^{(j)}_{s,r} \ne 0$ for some $1 \le j \le l, 2 \le s \le m$, and $1 \le r \le s-1$.
	Let $x = (x_1, \dots, x_m)$ so that $\{1, x_1, \dots, x_m\}$ is linearly independent over
	$\Q \left( \left\{ \alpha^{(j)}_s : 1 \le j \le l, 1 \le s \le m \right\} \right)$.
	
	Put
	$r_0 := \max \left\{ 1 \le r \le m-1 : a^{(j)}_{s,r} \ne 0~\text{for some}~1 \le j \le l~\text{and}~r+1 \le s \le m \right\}$,
	and let $S := \left\{ r_0 + 1 \le s \le m : a^{(j)}_{s,r_0} \ne 0~\text{for some}~1 \le j \le l \right\}$.
	For $s \in S$, let
	\begin{equation*}
		v_s := \left( a^{(1)}_{s,r_0}, \dots, a^{(l)}_{s,r_0} \right) \in \Z^l \setminus \{0\}.
	\end{equation*}
	
	We claim that without loss of generality, $\{v_s : s \in S\}$ is linearly independent over $\Q$.
	Indeed, suppose $\sum_{s \in S}{c_sv_s} = 0$ for some $(c_s)_{s \in S} \in \Z^S \setminus \{0\}$.
	Taking $s_0 := \max\{ s \in S : c_s \ne 0 \}$, we can perform a change of variables
	\begin{equation*}
		\tilde{x}_{s_0} := \sum_{s \in S}{c_sx_s}.
	\end{equation*}
	In the new coordinates, this gives $\tilde{v}_{s_0} = 0$, so $\tilde{S} = S \setminus \{s_0\}$
	and linear dependence is removed.
	Moreover, this change of variables is $|c_{s_0}|$-to-one, so well-distribution in the new coordinates
	implies well-distribution in the original system,
	since orbit closures of polynomial sequences must be finite unions of subtori
	(see Proposition \ref{prop: finite subtori}).
	
	Assume now that $\{v_s : s \in S\}$ is linearly independent over $\Q$.
	Let $(c_{i,s})_{1 \le i \le k, 1 \le s \le m} \in \Z^{mk} \setminus \{0\}$.
	If $c_{i,s} = 0$ for $1 \le i \le k$ and $s \in S$, then we can reduce to the lower-dimensional torus
	consisting of those coordinates not in the set $S$.
	Thus, we may assume $c_{i,s} \ne 0$ for some $1 \le i \le k$ and $s \in S$.
	Now expand
	\begin{equation*}
		c \cdot u_x(n) = \sum_{s=1}^m{P_s(n)x_s} + R(n),
	\end{equation*}
	where $P_s$ is $\Z$-valued for each $s = 1, \dots, m$ and $R(n)$
	is linearly independent from $\{x_1, \dots, x_s\}$ over $\Q$ for every $n \in \Z^d$.
	By the restriction on the coordinates of $x$, we can compute
	\begin{align*}
		P_{r_0}(n) & = \sum_{i=1}^k{c_{i,r_0}} + \sum_{i=1}^k{\sum_{s \in S}{c_{i,s} \left( v_s \cdot p_i(n) \right)}} \\
		 & = \sum_{i=1}^k{c_{i,r_0}} + \sum_{i=1}^k{\sum_{j=1}^l{\sum_{s \in S}{c_{i,s}a^{(j)}_{s,r_0}p_{i,j}(n)}}} \\
		 & = \text{const.} + \sum_{i=1}^k{\sum_{j=1}^l{d_{i,j}p_{i,j}(n)}},
	\end{align*}
	where
	\begin{equation*}
		\left( d_{i,1}, \dots, d_{i,l} \right) = \sum_{s \in S}{c_{i,s}v_s}
	\end{equation*}
	for $1 \le i \le k$.
	By assumption, $(c_{i,s})_{1 \le i \le k, s \in S} \ne 0$.
	Since $\{v_s : s \in S\}$ is linearly independent over $\Q$, this implies that $d_{i,j} \ne 0$
	for some $1 \le i \le k$ and $1 \le j \le l$.
	Therefore, $P_{r_0}(n)$ is nonconstant, since $\{p_{i,j} : 1 \le i \le k, 1 \le j \le l\}$ is an independent family.
	It follows that the polynomial $c \cdot u_x(n)$ has at least one irrational coefficient other than the constant term,
	so $c \cdot u_x(n)$ is well-distributed in $\T$ by Lemma \ref{lem: Weyl}.
	
	We have shown that $c \cdot u_x(n)$ is well-distributed in $\T$ for every $c \in \Z^{mk} \setminus \{0\}$.
	By Lemma \ref{lem: ud linear comb}, it follows that $u_x(n)$ is well-distributed in $\T^{mk}$ as desired. \\
	
	Let $E$ be the set of exceptional points $x \in \T^m$ such that
	$\left( u_x(n) \right)_{n \in \Z^d}$ is not well-distributed in $\T^{mk}$.
	The above argument show that if $x = (x_1, \dots, x_m) \in E$, then
	\begin{equation} \label{eq: linear relation}
		c_0 + \sum_{r=1}^m{c_r x_r} = 0
	\end{equation}
	for some coefficients
	\begin{equation*}
		(c_r)_{r=0}^m \in \Q \left( \left\{ \alpha^{(j)}_s : 1 \le j \le l, 1 \le s \le m \right\} \right)^{m+1} \setminus \{0\}.
	\end{equation*}
	For each such choice of coefficients $(c_r)_{r=0}^m$, the equation \eqref{eq: linear relation}
	defines a subtorus of dimension $m-1$.
	Hence, $E$ is contained in a countable union of $(m-1)$-dimensional subtori.
	In particular, $E$ is both a set of measure zero and meager in $\T^m$.
\end{proof}

%%%%%%%%%%%%%%%%%%%%%%%%%%%%%%%%%%%%%%%%%%%%%%%%%%%%

\subsubsection{The general case}

Theorem \ref{thm: tot erg} says the for totally ergodic systems, the trivial factor is characteristic
for independent polynomials $\{p_1, \dots, p_k\}$.
In particular, the Kronecker factor $\calZ = \calZ_1$ is characteristic.
Now, the collection of all families of independent $\O_K$-valued polynomials is clearly eligible,
so by Proposition \ref{prop: eligible}, the Kronecker factor is characteristic
for independent polynomials in any ergodic system.

In order to prove Theorem \ref{thm: rational Kronecker},
it remains only to reduce from the Kronecker factor to the rational Kronecker factor.
We want to prove: if $\E{f_i}{\calK_{rat}} = 0$ for some $i = 1, \dots, k$,
then $\UClim_{n \in \O_K}{\prod_{i=1}^k{T^{p_i(n)}f_i}} = 0$.
Since the Kronecker factor is spanned by eigenfunctions, we may assume that $f_i$ is an eigenfunction
with eigenvalue $\alpha_i \in \T^d$ for $i = 1, \dots, k$.
That is, $T^nf_i = e(n \cdot \alpha_i)f_i$.
The condition that $\E{f_i}{\calK_{rat}} = 0$ means that $\alpha_i \notin \Q^d$ for some $i = 1, \dots, k$.
Expanding the multiple ergodic average, we have
\begin{equation} \label{eq: eigenfunction avg}
	\UClim_{n \in \O_K}{\prod_{i=1}^k{T^{p_i(n)}f_i}}
	 = \UClim_{n \in \O_K}{e \left( \sum_{i=1}^k{\sum_{j=1}^d{p_{i,j}(n) \alpha_{i,j}}} \right) \prod_{i=1}^k{f_i}}.
\end{equation}
Since $\alpha_{i,j} \notin \Q$ for some $1 \le i \le k$ and $1 \le j \le d$, the polynomial
$\sum_{i=1}^k{\sum_{j=1}^d{p_{i,j}(n)\alpha_{i,j}}}$ has an irrational coefficient other than the constant term.
Thus, by Lemma \ref{lem: Weyl}, the average \eqref{eq: eigenfunction avg} is equal to 0 as desired.

%%%%%%%%%%%%%%%%%%%%%%%%%%%%%%%%%%%%%%%%%%%%%%%%%%%%

\subsection{Proof of Theorem \ref{thm: nilfactor}}

We follow the approach of Frantzikinakis (see \cite[Theorem A]{fra}),
modifying as necessary to upgrade to our multidimensional setting.

The polynomials $l_1p(n), \dots, l_kp(n)$ are essentially distinct, so the characteristic factor for the averages
\begin{equation} \label{eq: poly mult}
	\UClim_{n \in \O_K}{ \prod_{i=1}^k{T^{l_ip(n)}f_i} }
\end{equation}
is a nilfactor, $\calZ_r$, for some $r \in \Z$, by Theorem \ref{thm: polynomial nil}.
The content of Theorem \ref{thm: nilfactor} is thus to show that $r = k-1$.

We do this in several steps.
First, we will show that for totally ergodic systems, the limit \eqref{eq: poly mult} does not depend on $p$.
As a consequence, $\calZ_{k-1}$ is characteristic for totally ergodic systems,
since it is characteristic when $p(n) = n$.
We then apply Proposition \ref{prop: eligible} to conclude that $\calZ_{k-1}$ is characteristic in any ergodic system.

%%%%%%%%%%%%%%%%%%%%%%%%%%%%%%%%%%%%%%%%%%%%%%%%%%%%

\subsubsection{Totally ergodic systems}

For this section, we will assume that $\left( X, \B, \mu, T \right)$ is a totally ergodic $\O_K$-system,
and we set out to prove that the limit \eqref{eq: poly mult} is independent of the choice of polynomial $p$.

By Theorem \ref{thm: polynomial nil} and a standard approximation argument, we may further assume
that $X = G/\Gamma$ is a nilmanifold and $T$ is an action by niltranslations $T^nx = a(n)x$ with $a(n) \in G$.
It therefore suffices to show that the orbits
\begin{equation*}
	\left\{\left( T^{l_1p(n)}x, \dots, T^{l_kp(n)}x \right) : n \in \O_K\right\} \qquad \text{and} \qquad
	 \left\{\left( T^{l_1n}x, \dots, T^{l_kn}x \right) : n \in \O_K\right\}
\end{equation*}
are equidistributed for almost every $x \in X$.
Equivalently, letting $g(n) := \left( a(l_1n), \dots, a(l_kn) \right) \in G^k$ and $\tilde{x} = (x, \dots, x) \in X^k$,
we want to show that $\left\{ g(p(n))\tilde{x} : n \in \O_K \right\}$ and $\left\{ g(n)\tilde{x} : n \in \O_K \right\}$
are equidistributed for almost every $x \in X$.
Now, by Theorem \ref{thm: proj}, it is enough to show that these sequences have the same closure in $X^k$.

By Proposition \ref{prop: alg ind coord}, any nonconstant polynomial $p \in K[x]$
has algebraically independent coordinates, so we will prove a related result about $\Z^d$-valued polynomials
with algebraically independent coordinates:

\begin{prop}[cf. \cite{fra}, Proposition 2.7] \label{prop: polynomial orbit}
	Let $X = G/\Gamma$ be a nilmanifold, $g : \Z^l \to G$ a polynomial sequence, and $x \in X$.
	Suppose $p : \Z^d \to \Z^l$ is a polynomial with algebraically independent coordinates.
	If $Y := \overline{\left\{ g(n)x : n \in \Z^l \right\}}$ is connected,
	then $\overline{\left\{ g(p(n))x : n \in \Z^d \right\}} = Y$.
\end{prop}

\begin{proof}
	By Theorem \ref{thm: nil orbits}, $Y$ is a subnilmanifold $H/\Delta$.
	Now by Theorem \ref{thm: proj}, we may replace $H$ by $H/[H_0,H_0]$
	and assume that $H_0$ is abelian.
	As in Lemma \ref{lem: affine}, we may further reduce to the case that $Y = \T^m$
	and $g(n)x = T_1^{p_1(n)} \cdots T_k^{p_k(n)}x$ with $T_i$ unipotent affine actions.
	The coordinates of $\left\{ g(n)x : n \in \Z^l \right\}$ are real polynomials in $n \in \Z^l$,
	so it remains to show:
	if $u : \Z^l \to \T^m$ is a sequence with polynomial coordinates
	and $\overline{\left\{ u(n) : n \in \Z^l \right\}} = \T^m$,
	then $\overline{\left\{ u(p(n)) : n \in \Z^d \right\}} = \T^m$
	for every polynomial $p : \Z^d \to \Z^l$ with algebraically independent coordinates.
	
	The polynomial $u(n)$ can be decomposed as
	$u(n) = u(0) + u_0(n)q + u_1(n)\alpha_1 + \cdots + u_r(n)\alpha_r$, where $q \in \Q$,
	$\alpha_1, \dots, \alpha_r \in \R$ are linearly independent irrational numbers and $u_i : \Z^l \to \Z^m$
	are polynomials with $u_i(0) = 0$.
	By Corollary \ref{cor: span}, the orbit $\left\{ u(n) : n \in \Z^l \right\}$ is dense in $\T^m$ if and only if
	\begin{equation*}
		\spn{(u_1)} + \cdots + \spn{(u_r)} = \R^m.
	\end{equation*}
	Thus, it suffices to prove $\spn(u_i \circ p) = \spn(u_i)$ for each $i = 1, \dots, r$.
	
	Fix $1 \le i \le r$.
	Suppose the coordinates of $u_i \circ p$ satisfy a linear relation $\sum_{j=1}^m{c_ju_{i,j}(p(n))} = 0$
	for some $c_1, \dots, c_m \in \Z$, where $u_i = (u_{i,1}, \dots, u_{i,m})$ with $u_{i,j} : \Z^l \to \Z$.
	Let $v : \Z^l \to \Z$ be the polynomial $v(n) := \sum_{j=1}^m{c_ju_{i,j}(n)}$.
	Then $v \circ p = 0$.
	But the coordinates of $p$ are algebraically independent, so we must have $v = 0$.
	That is, the coordinates of $u_i$ satisfies the the same linear relation.
	Therefore, $\spn(u_i \circ p) = \spn(u_i)$ as desired.
\end{proof}

It remains only to show that $Y := \overline{\left\{ g(n) \tilde{x} : n \in \O_K \right\}}$ is connected.
This is where we use that the system is totally ergodic.
Let $Y_w = Hx_w$, $w \in W$, as in Theorem \ref{thm: nil orbits}.
Since $W$ is a finite group, $\omega^{-1}(0) \subseteq \Z^d$ has finite index in $\Z^d$.
Because $T$ is totally ergodic, we therefore have $Y_0 = Y$, so $Y$ is indeed connected.

In summary, we have shown the following:

\begin{thm} \label{thm: poly linear tot erg}
	Let $K$ be a number field with ring of integers $\O_K$.
	Let $\left( X, \B, \mu, T \right)$ be a totally ergodic $\O_K$-system.
	Let $p(x) \in K[x]$ be a non-constant $\O_K$-valued polynomial.
	Let $l_1, \dots, l_k \in \O_K$ be distinct and nonzero.
	Then
	\begin{equation*}
		\UClim_{n \in \O_K}{ \prod_{i=1}^k{T^{l_ip(n)}f_i} }
		 = \UClim_{n \in \O_K}{ \prod_{i=1}^k{T^{l_in}f_i} }.
	\end{equation*}
\end{thm}

%%%%%%%%%%%%%%%%%%%%%%%%%%%%%%%%%%%%%%%%%%%%%%%%%%%%

\subsubsection{General case}

Now we prove Theorem \ref{thm: nilfactor}.
Letting $l := \gcd(l_1, \dots, l_k)$ and replacing $l_1, \dots, l_k$ by $l'_i := \frac{l_i}{l}$ and $p$ by $lp$,
we may assume without loss of generality that $l = 1$.
By \cite[Theorem 4.1.2]{griesmer}, the characteristic factor for $\{l_1n, \dots, l_kn\}$ is $\calZ_{k-1}$.
Thus, by Theorem \ref{thm: poly linear tot erg}, $\calZ_{k-1}$ is characteristic for $\{l_1p(n), \dots, l_kp(n)\}$
in the case of totally ergodic systems.
It is easily checked that the collection
\begin{equation*}
	\P := \left\{ \{l_1p(n), \dots, l_kp(n)\} : p(x) \in K[x]~\text{is noncontant and}~\O_K\text{-valued} \right\}
\end{equation*}
is eligible (see Definition \ref{defn: eligible}) under the assumption that $l = \gcd(l_1, \dots, l_k) = 1$.
Hence, Theorem \ref{thm: nilfactor} follows by Proposition \ref{prop: eligible}.

%%%%%%%%%%%%%%%%%%%%%%%%%%%%%%%%%%%%%%%%%%%%%%%%%%%%

\section{Large intersections} \label{sec: Khintchine}

Having established characteristic factors for the polynomial multiple ergodic averages of interest,
we now move to deducing the related Khintchine-type theorems.

%%%%%%%%%%%%%%%%%%%%%%%%%%%%%%%%%%%%%%%%%%%%%%%%%%%%

\subsection{Proof of Theorem \ref{thm: IndLargeInt}} \label{sec: IndLargeInt}

We want to prove Theorem \ref{thm: IndLargeInt}, restated here for the convenience of the reader:

\IndLargeInt*

\noindent We will prove the following stronger statement:

\begin{thm} \label{thm: large limit}
	Let $K$ be a number field with ring of integers $\O_K$.
	Suppose $\{p_1, \dots, p_k\} \subseteq K[x]$ is a jointly intersective family
	of linearly independent $\O_K$-valued polynomials.
	Then for any measure-preserving $\O_K$-system $\left( X, \B, \mu, T \right)$, $A \in \B$, and $\eps > 0$,
	there exist $\xi \in \O_K$ and $D \in \O_K \setminus \{0\}$ such that
	\begin{equation*}
		\UClim_{n \in \O_K}{\mu \left( A \cap T^{-p_1(\xi + Dn)}A \cap \cdots \cap T^{-p_k(\xi + Dn)}A \right)}
		 > \mu(A)^{k+1} - \eps.
	\end{equation*}
\end{thm}

\noindent Assuming Theorem \ref{thm: large limit}, the set
\begin{equation*}
	\left\{ n \in \O_K : \mu \left( A \cap T^{-p_1(\xi + Dn)}A \cap \cdots \cap T^{-p_k(\xi + Dn)}A \right)
	 > \mu(A)^{k+1} - \eps \right\}
\end{equation*}
is syndetic by Proposition \ref{prop: Folner syndetic}.
Since $\xi + D\O_K$ is syndetic in $\O_K$, Theorem \ref{thm: IndLargeInt} follows immediately.

\begin{proof}[Proof of Theorem \ref{thm: large limit}]
	First assume that $T$ is ergodic.
	The rational Kronecker factor $\calK_{rat}$ is the inverse limit of the periodic factors
	$\calK_r := \left\{ f \in L^2(\mu) : T^{rn}f = f~\text{for all}~n \in \O_K \right\}$, $r \in \O_K$.
	Note that $\calK_r \subseteq \calK_s$ if $r \mid s$ in $\O_K$.
	Thus, we may approximate $\calK_{rat}$ by $\calK_r$ for some $r \in \O_K$.
	To be precise, there exists $r \in \O_K$ such that
	\begin{equation*}
		\norm{1}{\E{\ind_A}{\calK_{rat}} - \E{\ind_A}{\calK_r}} < \frac{\eps}{k+1}.
	\end{equation*}
	
	Now, the system $\left( X, \B, \mu, (T^{rn})_{n \in \O_K} \right)$ has finitely many ergodic components.
	In fact, for some $m \le [\O_K : r\O_K]$, $X$ can be partitioned into $m$ disjoint sets
	$X_1, \dots, X_m \in \B$ with $\mu(X_j) = \frac{1}{m}$
	such that $\mu \left( X_j \triangle T^{-rn}X_j \right) = 0$ and $\left( X, \B, \mu_j, (T^{rn})_{n \in \O_K} \right)$
	is ergodic, where $\mu_j(B) = m \cdot \mu(B \cap X_j)$.
	
	By Lemma \ref{lem: intersective}, let $\xi \in \O_K$ and $D \in \O_K \setminus \{0\}$
	such that $p_i(\xi + D\O_K) \subseteq r\O_K$ for $i = 1, \dots, k$.
	For each $i = 1, \dots, k$, let $q_i(x) \in K[x]$ be the $\O_K$-valued polynomial $q_i(x) := r^{-1} p_i(\xi + Dx)$.
	Then by Theorem \ref{thm: rational Kronecker},	
	\begin{align*}
		\UClim_{n \in \O_K}&{\mu \left( A \cap T^{-p_1(\xi + Dn)}A \cap \cdots \cap T^{-p_k(\xi + Dn)}A \right)} \\
		 & = \UClim_{n \in \O_K}{\frac{1}{m} \sum_{j=1}^m{
		 \mu_j \left( A \cap (T^r)^{-q_1(n)}A \cap \cdots \cap (T^r)^{-q_k(n)}A \right)}} \\
		 & = \UClim_{n \in \O_K}{\frac{1}{m} \sum_{j=1}^m{
		 \int_X{\E{\ind_A}{\calK_{rat}} \prod_{i=1}^k{(T^r)^{q_i(n)}\E{\ind_A}{\calK_{rat}}}~d\mu_j}}} \\
		 & = \UClim_{n \in \O_K}{
		 \int_X{\E{\ind_A}{\calK_{rat}} \prod_{i=1}^k{T^{rq_i(n)}\E{\ind_A}{\calK_{rat}}}~d\mu}} \\
		 & > \UClim_{n \in \O_K}{
		 \int_X{\E{\ind_A}{\calK_r} \prod_{i=1}^k{T^{rq_i(n)}\E{\ind_A}{\calK_r}}~d\mu}} - \eps \\
		 & = \int_X{ \left( \E{\ind_A}{\calK_r} \right)^{k+1}~d\mu} - \eps \\
		 & \ge \left( \int_X{\E{\ind_A}{\calK_r}~d\mu} \right)^{k+1} - \eps \\
		 & = \mu(A)^{k+1} - \eps.
	\end{align*}
	
	\bigskip
	
	Now suppose $T$ is not ergodic.
	Let $\mu = \int_{\Omega}{\mu_{\omega}~d\rho(\omega)}$ be the ergodic decomposition.
	For each $\omega \in \Omega$, let $r_{\omega} \in \O_K$ be minimal (with respect to divisibility) so that
	\begin{equation*}
		\norm{L^1(\mu_{\omega})}{\E{\ind_A}{\calK_{rat}(\mu_{\omega})}
		 - \E{\ind_A}{\calK_{r_\omega}}(\mu_{\omega})} < \frac{\eps}{2(k+1)}.
	\end{equation*}
	The function $\omega \mapsto r_{\omega}$ is measurable,
	so we may define $\Omega_r := \{\omega \in \Omega : r_{\omega} \mid r\}$
	and let $\mu_r := \int_{\Omega_r}{\mu_{\omega}~d\rho(\omega)}$.
	Then let $r \in \O_K$ so that $\rho(\Omega \setminus \Omega_r) < \frac{\eps}{2}$.
	
	Note that in the proof of the ergodic case, the numbers $\xi$ and $D$ depend only on $r$ and not on $\mu$.
	Thus, for every $\omega \in \Omega_r$, we have
	\begin{equation*}
		\UClim_{n \in \O_K}{\mu_{\omega} \left( A \cap T^{-p_1(\xi + Dn)}A \cap \cdots \cap T^{-p_k(\xi + Dn)}A \right)}
		 > \mu_{\omega}(A)^{k+1} - \frac{\eps}{2}.
	\end{equation*}
	Now we integrate over $\Omega$:
	\begin{align*}
		\UClim_{n \in \O_K}&{\mu \left( A \cap T^{-p_1(\xi + Dn)}A \cap \cdots \cap T^{-p_k(\xi + Dn)}A \right)} \\
		 & \ge \UClim_{n \in \O_K}{\int_{\Omega_r}{
		 \mu_{\omega} \left( A \cap T^{-p_1(\xi + Dn)}A \cap \cdots \cap T^{-p_k(\xi + Dn)}A \right)~d\rho(\omega)}} \\
		 & > \int_{\Omega_r}{\left( \mu_{\omega}(A)^{k+1} - \frac{\eps}{2} \right)~d\rho(\omega)} \\
		 & \ge \int_{\Omega_r}{\mu_{\omega}(A)^{k+1}~d\rho(\omega)} - \frac{\eps}{2} \\
		 & > \int_{\Omega}{\mu_{\omega}(A)^{k+1}~d\rho(\omega)} - \eps \\
		 & \ge \left( \int_{\Omega}{\mu_{\omega}(A)~d\rho(\omega)} \right)^{k+1} - \eps \\
		 & = \mu(A)^{k+1} - \eps.
	\end{align*}
\end{proof}

%%%%%%%%%%%%%%%%%%%%%%%%%%%%%%%%%%%%%%%%%%%%%%%%%%%%

\subsection{Proof of Theorem \ref{thm: MultLargeInt}} \label{sec: MultLargeInt}

Now we turn to proving Theorem \ref{thm: MultLargeInt}, restated below:

\MultLargeInt*

First we will prove the special case when $T$ is totally ergodic.
In this case, by applying Theorem \ref{thm: nilfactor}, we can compute limits explicitly:

\begin{thm} \label{thm: limit formulae}
	Let $K$ be a number field with ring of integers $\O_K$.
	Let $\X = (X, \B, \mu, T)$ be a totally ergodic $\O_K$-system.
	Let $Z$ be a compact abelian group and $\alpha : (\O_K, +) \to Z$ a homomorphism
	such that the Kronecker factor of $\X$ is isomorphic to the system $\mathbf{Z} = (Z, \B_Z, \mu_Z, S)$,
	where $\B_Z$ is the Borel $\sigma$-algebra, $\mu_Z$ is the Haar probability measure,
	and $S$ acts by rotations $S^nz = z + \alpha_n$ for $n \in \O_K$.
	\begin{enumerate}[1.]
		\item	Let $r, s \in \O_K$ distinct and nonzero,
			$p(x) \in K[x]$ an $\O_K$-valued polynomial, and $f_1, f_2 \in L^{\infty}(\mu)$.
			Then
			\begin{equation} \label{eq: limit formula 2}
				\UClim_{n \in \O_K}{T^{rp(n)}f_1(x) \cdot T^{sp(n)}f_2(x)}
				 = \int_{Z^2}{\tilde{f_1}(z+u) \tilde{f_2}(z+v)~d\nu(u,v)}
			\end{equation}
			in $L^2(\mu)$, where $x \mapsto z$ is the factor map, $\tilde{f} = \E{f}{Z}$,
			and $\nu$ is the Haar measure on the subgroup
			$\overline{\left\{ \left( \alpha_{rn}, \alpha_{sn} \right) : n \in \O_K \right\}} \subseteq Z^2$.
		\item	Let $a_1, a_2 \in \Z \setminus \{0\}$ be coprime, and put $a_3 = a_1 + a_2$.
			There is a compact abelian group $H$ such that the nilfactor $(X, \calZ_2, \mu, T)$
			is isomorphic to a skew-product system $\mathbf{Z} \times_{\sigma} H$,
			and there exists a function $\psi : Z^2 \to H$ such that $\psi(0,\cdot) = 0$ and
			$t \mapsto \psi(t, \cdot)$ is continuous as a function from $Z$ to the space $\mathcal{M}(Z,H)$
			of measurable functions $Z \to H$ in the topology of convergence in measure,
			and integers $b_1, b_2, b_3 \in \Z$ such that:
			for any $\O_K$-valued polynomial $p(x) \in K[x]$ and any $f_1, f_2, f_3 \in L^{\infty}(\mu)$, we have
			\begin{equation} \label{eq: limit formula 3}
				\UClim_{n \in \O_K}{\prod_{i=1}^3{T^{a_ip(n)}f_i(x)}}
				 = \int_{Z \times H^2}{
				 \prod_{i=1}^3{\tilde{f}_i(z + a_it, h + a_iu + a_i^2v + b_i \psi(t,z))}~dt~du~dv}
			\end{equation}
			in $L^2(\mu)$, where $\tilde{f} = \E{f}{\calZ_2}$.
	\end{enumerate}
\end{thm}
\begin{proof}
	Since the system $\X$ is totally ergodic, the limits
	\begin{equation*}
		\UClim_{n \in \O_K}{T^{rp(n)}f_1 \cdot T^{sp(n)}f_2}
		\quad \text{and} \quad
		\UClim_{n \in \O_K}{\prod_{i=1}^3{T^{a_ip(n)}f_i}}
	\end{equation*}
	are independent of the choice of the polynomial $p$ by Theorem \ref{thm: nilfactor}.
	Thus, we may assume without loss of generality that $p(n) = n$.
	The identity \eqref{eq: limit formula 2} is then a special case of \cite[Theorem 3.1]{abb},
	and \eqref{eq: limit formula 3} is a special case of \cite[Theorem 7.1]{abb}.
\end{proof}

\begin{cor} \label{cor: twisted limit formulae}
	Let $K$ be a number field with ring of integers $\O_K$.
	Let $\X = (X, \B, \mu, T)$ be a totally ergodic $\O_K$-system with Kronecker factor $(\mathbf{Z}, \alpha)$.
	
	\begin{enumerate}[1.]
		\item	Let $r, s \in \O_K$ distinct and nonzero,
			$p(x) \in K[x]$ an $\O_K$-valued polynomial, and $f_0, f_1, f_2 \in L^{\infty}(\mu)$.
			Then for any continuous function $\eta : Z^2 \to \C$, we have
			\begin{equation} \label{eq: twisted limit formula 2} \begin{split}
				\UClim_{n \in \O_K}&{\eta \left( \alpha_{rp(n)}, \alpha_{sp(n)} \right)
				 \int_X{f_0 \cdot T^{rp(n)}f_1 \cdot T^{sp(n)}f_2~d\mu}} \\
				 & = \int_{Z^3}{\eta(u,v) \tilde{f_0}(z) \tilde{f_1}(z+u) \tilde{f_2}(z+v)~dz~d\nu(u,v)}
			\end{split} \end{equation}
			in $L^2(\mu)$, where $x \mapsto z$ is the factor map, $\tilde{f} = \E{f}{Z}$,
			and $\nu$ is the Haar measure on the subgroup
			$\overline{\left\{ \left( \alpha_{rn}, \alpha_{sn} \right) : n \in \O_K \right\}} \subseteq Z^2$.
		\item	Let $a_1, a_2 \in \Z \setminus \{0\}$ be coprime, and put $a_0 = 0, a_3 = a_1 + a_2$.
			Let $H$, $\psi$, and $b_i$ be as in Theorem \ref{thm: limit formulae}(2).
			Let $p(x) \in K[x]$ be an $\O_K$-valued polynomial, and let $f_0, f_1, f_2, f_3 \in L^{\infty}(\mu)$.
			Then for any continuous function $\eta : Z \to \C$,
			\begin{equation} \label{eq: twisted limit formula 3} \begin{split}
				\UClim_{n \in \O_K}&{\eta \left( \alpha_{p(n)} \right)
				 \int_X{\prod_{i=0}^3{T^{a_ip(n)}f_i}~d\mu}} \\
				 & = \int_{Z^2 \times H^3}{\eta(t)
				 \prod_{i=0}^3{\tilde{f}_i(z + a_it, h + a_iu + a_i^2v + b_i \psi(t,z))}~dz~dt~dh~du~dv}
			\end{split} \end{equation}
			in $L^2(\mu)$, where $\tilde{f} = \E{f}{\calZ_2}$.
	\end{enumerate}
\end{cor}

\begin{proof}
	(1) Since $Z^2$ is a compact abelian group, we may assume by the Stone--Weierstrass theorem
	that $\eta(u,v) = \lambda_1(u)\lambda_2(v)$ for $u, v \in Z$, where $\lambda_1, \lambda_2 \in \hat{Z}$.
	Defining
	\begin{align*}
		g_0(x) & = \overline{\lambda_1(z)} \overline{\lambda_2(z)} f_0(x) \\
		\intertext{and}
		g_i(x) & = \lambda_i(z) f_i(x)
	\end{align*}
	for $i = 1, 2$, the formula \eqref{eq: twisted limit formula 2} then follows
	by applying \eqref{eq: limit formula 2} to the functions $g_1, g_2$ and integrating against $g_0$. \\
	
	(2) Again, without loss of generality, we may assume $\eta = \lambda \in \hat{Z}$.
	Since $\gcd(a_1, a_2) = 1$, there are integers $c_1, c_2 \in \Z$ so that $c_1a_1 + c_2a_2 = 1$.
	Let $c_3 = 0$ and $c_0 = -(c_1+c_2)$ so that
	\begin{equation*}
		\sum_{i=0}^3{c_i} = 0 \quad \text{and} \quad
		\sum_{i=0}^3{c_ia_i} = 1.
	\end{equation*}
	Then define $g_i(x) := \lambda(c_iz) f_i(x)$ for $i = 0, 1, 2, 3$.
	Applying the formula \eqref{eq: limit formula 3} for the functions $g_1, g_2, g_3$ and integrating against $g_0$
	produces the desired formula \eqref{eq: twisted limit formula 3}.
\end{proof}

\begin{prop} \label{prop: MultLargeInt tot erg}
	Let $K$ be a number field with ring of integers $\O_K$.
	Let $(X, \B, \mu, T)$ be a totally ergodic $\O_K$-system, $r, s \in \O_K$ distinct and nonzero,
	and $p(x) \in K[x]$ an $\O_K$-valued polynomial.
	Then for any $A \in \B$ with $\mu(A) > 0$ and any $\eps > 0$, the set
	\begin{equation*}
		\left\{ n \in \O_K :
		 \mu \left( A \cap T^{-rp(n)}A \cap T^{-sp(n)}A \right) > \mu(A)^3 - \eps \right\}
	\end{equation*}
	is syndetic.
	
	Moreover, if $\frac{s}{r} \in \Q$, then
	\begin{equation*}
		\left\{ n \in \O_K :
		 \mu \left( A \cap T^{-rp(n)}A \cap T^{-sp(n)}A \cap T^{-(r+s)p(n)}A \right) > \mu(A)^4 - \eps \right\}
	\end{equation*}
	is syndetic.
\end{prop}

\begin{rem}
	We do not assume that the polynomial $p$ is intersective in Proposition \ref{prop: MultLargeInt tot erg}.
	This is because, in the totally ergodic setting, there are no ``local obstructions'' that need to be avoided.
	In order to extend to the ergodic setting, however, we will have to restrict to intersective polynomials.
\end{rem}

\begin{proof}[Proof of Proposition \ref{prop: MultLargeInt tot erg}]
	We adapt the method from \cite{fra}. \\
	
	First we prove the double recurrence result.
	Using the formula \eqref{eq: twisted limit formula 2} with $f_i = \ind_A$
	and choosing $\eta$ supported on a small neighborhood of 0, it suffices to show
	\begin{equation*}
		\int_Z{(\E{\ind_A}{\calZ})^3~dz} \ge \mu(A)^3.
	\end{equation*}
	But this follows immediately from Jensen's inequality, so
	\begin{equation*}
		\left\{ n \in \O_K :
		 \mu \left( A \cap T^{-rp(n)}A \cap T^{-sp(n)}A \right) > \mu(A)^3 - \eps \right\}
	\end{equation*}
	is syndetic. \\
	
	Now we move to triple recurrence.
	Since $\frac{s}{r} \in \Q$, we can write $r = a_1k$ and $s = a_2k$ for some coprime $a_1, a_2 \in \Z$
	and some $k \in K$.
	Let $q(n) = kp(n)$.
	Note that $a_1q(n) = rp(n)$ and $a_2q(n) = sp(n)$ are $\O_K$-valued.
	Therefore, $q$ is itself $\O_K$-valued, since $\gcd(a_1, a_2) = 1$.
	Hence, without loss of generality, we will assume that $r$ and $s$ are coprime integers.
	
	Now put $a_0 = 0$, $a_1 = r$, $a_2 = s$, and $a_3 = r + s$.
	Applying formula \eqref{eq: twisted limit formula 3} with $f_i = \ind_A$
	and choosing the function $\eta$ to be supported on a small neighborhood of 0,
	we want to show
	\begin{equation} \label{eq: triple rec ineq}
		\int_{Z \times H^3}{\prod_{i=0}^3{\E{\ind_A}{\calZ_2}(z, h + a_iu + a_i^2v)}~dh~du~dv~dz} \ge \mu(A)^4.
	\end{equation}
	
	Fix $z \in Z$, and let $F_z : H \to [0,1]$ be the function $F_z(x) = \E{\ind_A}{\calZ_2}(z,x)$.
	Now we perform several changes of variables.
	First, take $h = a_3x$:
	\begin{align*}
		& \int_{H^3}{\prod_{i=0}^3{F_z(h + a_iu + a_i^2v)}~dh~du~dv} \\
		 & = \int_{H^3}{F_z(a_3x) F_z\left( a_3\left( x + u + a_3v \right) \right)
		 F_z(a_3x + a_1u + a_1^2v) F_z(a_3x + a_2u + a_2^2v)~du~dx~dv} \\
	\intertext{Next, $x + u + a_3v = y$:}
		 & = \int_{H^3}{F_z(a_3x) F_z(a_3y) F_z(a_2x + a_1y - a_1a_2v) F_z(a_1x + a_2y - a_1a_2v)~dv~dx~dy} \\
	\intertext{Now, $a_1(x+y) - a_1a_2v = w$:}
		 & = \int_{H^3}{F_z(a_3x) F_z(a_3y) F_z((a_2-a_1)x + w) F_z((a_2-a_1)y + w)~dx~dy~dw} \\
		 & = \int_H{\left( \int_H{F_z(a_3x) F_z \left( (a_2 - a_1)x + w \right)~dx}\right)^2~dw} \\
	\intertext{Apply Jensen's inequality:}
		 & \ge \left( \int_{H^2}{F_z(a_3x) F_z \left( (a_2 - a_1)x + w \right)~dw~dx} \right)^2 \\
	\intertext{Finally, let $w + (a_2-a_1)x = u$ and $a_3x = t$:}
		 & = \left( \int_H{F_z(t)~dt} \right)^2 \left( \int_H{F_z(u)~du} \right)^2 \\
		 & = \left( \int_H{F_z~dm_H} \right)^4.
	\end{align*}
	
	Thus, applying Jensen's inequality one more time, we have
	\begin{align*}
		\int_{Z \times H^3}&{\prod_{i=0}^3{\E{\ind_A}{\calZ_2}(z, h + a_iu + a_i^2v)}~dh~du~dv~dz} \\
		 &= \int_Z{\left( \int_{H^3}{\prod_{i=0}^3{F_z(h + a_iu + a_i^2v)}~dh~du~dv} \right)~dz} \\
		 & \ge \int_Z{ \left( \int_H{F_z~dm_H} \right)^4~dz} \\
		 & \ge \left( \int_Z{\int_H{F_z~dm_H}~dz} \right)^4 \\
		 & = \mu(A)^4.
	\end{align*}
	
	That is, the inequality \eqref{eq: triple rec ineq} holds, so the set
	\begin{equation*}
		\left\{ n \in \O_K :
		 \mu \left( A \cap T^{-rp(n)}A \cap T^{-sp(n)}A \cap T^{-(r+s)p(n)}A \right) > \mu(A)^4 - \eps \right\}
	\end{equation*}
	is syndetic.
\end{proof}

We have proved Theorem \ref{thm: MultLargeInt} in the case when $T$ is totally ergodic.
We will now extend this to the general case that $T$ is simply ergodic.
Theorem \ref{thm: nilfactor} still applies, so by a standard approximation argument,
we may assume without loss of generality that $T$ acts by niltranslations.
The Kronecker factor is then a group of the form $\Z_{a_1} \times \cdots \times \Z_{a_d} \times \T^{c}$.
As in the proof of Proposition \ref{prop: eligible}, we can therefore find $k \in \O_K$
such that the Kronecker factor of $\left( T^{kn} \right)_{n \in \O_K}$ is connected,
and hence each of the finitely many ergodic components of $\left( T^{kn} \right)_{n \in \O_K}$ is totally ergodic
by Proposition \ref{prop: tot erg conn}.
Let $X_1, \dots, X_m$ the atoms of the $\left( T^{kn} \right)_{n \in \O_K}$-invariant $\sigma$-algebra,
and let $\mu_j(B) := m \cdot \mu(B \cap X_j)$ so that $\mu$
has ergodic decomposition $\mu = \frac{1}{m} \sum_{j=1}^m{\mu_j}$ for the action $(T^{kn})_{n \in \O_K}$.

By Lemma \ref{lem: intersective}, let $\xi \in \O_K$ and $D \in \O_K \setminus \{0\}$
so that $p(\xi + D\O_K) \subseteq k\O_K$.
Let $q(x) \in K[x]$ be the $\O_K$-valued polynomial $q(n) = k^{-1}p(\xi + Dn)$ for every $n \in \O_K$.
Following the argument in the proof of Proposition \ref{prop: MultLargeInt tot erg},
we can choose a continuous function $\eta$ concentrated on a sufficiently small neighborhood of $0$ in $Z^2$
with $\int_{Z^2}{\eta~d\nu} = 1$ so that
\begin{equation*}
	\UClim_{n \in \O_K}{\eta(\alpha_{rq(n)}, \alpha_{sq(n)}) \mu_j \left( A \cap T^{-krq(n)}A \cap T^{-ksq(n)}A \right)}
	 \ge \mu_j(A)^3
\end{equation*}
for $j = 1, \dots, m$.
Summing over $j = 1, \dots, m$ and applying Jensen's inequality, we get
\begin{equation*}
	\UClim_{n \in \O_K}{\eta(\alpha_{rq(n)}, \alpha_{sq(n)}) \mu \left( A \cap T^{-krq(n)}A \cap T^{-ksq(n)}A \right)}
	 \ge \mu(A)^3
\end{equation*}
from which it follows that
\begin{equation*}
	\left\{ n \in \O_K : \mu \left( A \cap T^{-krq(n)}A \cap T^{-ksq(n)}A \right) > \mu(A)^3 - \eps \right\}
\end{equation*}
is syndetic in $\O_K$.

A similar argument with the ergodic decomposition can be used to show that, if $\frac{s}{r} \in \Q$, then
\begin{equation*}
	\left\{ n \in \O_K :
	 \mu \left( A \cap T^{-krq(n)}A \cap T^{-ksq(n)}A \cap T^{-k(r+s)q(n)}A \right) > \mu(A)^4 - \eps \right\}
\end{equation*}
is also syndetic.

Thus, the sets
\begin{equation*}
	\left\{ n \in \O_K :
	 \mu \left( A \cap T^{-rp(n)}A \cap T^{-sp(n)}A \right) > \mu(A)^3 - \eps \right\}
\end{equation*}
and (if $\frac{s}{r} \in \Q$)
\begin{equation*}
	\left\{ n \in \O_K :
	 \mu \left( A \cap T^{-rp(n)}A \cap T^{-sp(n)}A \cap T^{-(r+s)p(n)}A \right) > \mu(A)^4 - \eps \right\}
\end{equation*}
are relatively syndetic in $\xi + D\O_K$.
But $\xi + D\O_K$ is syndetic in $\O_K$, so we are done.

%%%%%%%%%%%%%%%%%%%%%%%%%%%%%%%%%%%%%%%%%%%%%%%%%%%%
%%%%%%%%%%%%%%%%%%%%%%%%%%%%%%%%%%%%%%%%%%%%%%%%%%%%

\section{Refinements} \label{sec: refinements}

%%%%%%%%%%%%%%%%%%%%%%%%%%%%%%%%%%%%%%%%%%%%%%%%%%%%

\subsection{Polynomial IP sets}

Recall that a set $E \subseteq \O_K$ is $\text{IP}^*$ if it intersects every finite sum set
\begin{equation*}
	FS \left( (x_n)_{n \in \N} \right) := \left\{ \sum_{n \in F}{x_n} : F \subseteq \N~\text{is finite and nonempty} \right\},
\end{equation*}
where $(x_n)_{n \in \N}$ is a sequence of distinct elements of $\O_K$.
Similarly, we say $E$ is \emph{$\text{IP}_r^*$} if it intersects every finite sum set of the form
\begin{equation*}
	FS(x_1, \dots, x_r) := \left\{ \sum_{k=1}^s{x_{n_k}} : 1 \le s \le r, n_1 < n_2 < \dots < n_s \right\},
\end{equation*}
where $x_1, \dots, x_r \in \O_K$ are distinct and nonzero.
Finally, $E$ is called an \emph{$\text{IP}_0^*$ set} if $E$ is $\text{IP}_r^*$ for some $r \in \N$.
Clearly, every $\text{IP}_0^*$ set is also $\text{IP}^*$, but the converse is not true.

Now we will define polynomial generalizations of IP and $\text{IP}_0$ sets.
For a set $S$, let $\calF(S)$ denote the semigroup of finite subsets of $S$ with the union operation.

\begin{defn}
	Let $(H,+)$ be an abelian group, and let $\varphi : \calF(S) \to H$.
	\begin{enumerate}[1.]
		\item	We say that $\varphi$ is \emph{linear} if
			$\varphi(\alpha \cup \beta) = \varphi(\alpha) + \varphi(\beta)$ whenever $\alpha \cap \beta = \es$.
		\item	For $\beta \in \calF(S)$, the \emph{$\beta$-derivate} of $\varphi$
			is the function $D_{\beta}\varphi : \calF(S \setminus \beta) \to H$ given by
			$D_{\beta}\varphi(\alpha) = \varphi(\alpha \cup \beta) - \varphi(\alpha)$.
		\item	We say $\varphi$ is a \emph{polynomial of degree $\le d$} if
			for any disjoint sets $\beta_0, \dots, \beta_d \in \calF(S)$, one has
			$D_{\beta_0} D_{\beta_1} \cdots D_{\beta_d}\varphi = 0$.
	\end{enumerate}
\end{defn}

Note that an IP set has the form $\{\varphi(\alpha) : \alpha \in \calF(\N), \alpha \ne \es\}$
for a linear mapping $\varphi : \calF(S) \to \O_K$ with $\varphi(\es) = 0$.
For a polynomial mapping $\varphi : \calF(S) \to \O_K$, we call the corresponding set
$\{\varphi(\alpha) : \alpha \in \calF(\N), \alpha \ne \es\}$ a \emph{VIP set}.
Similarly, if $\varphi : \calF(\{1, \dots, r\} \to \O_K$ is a polynomial mapping of degree $\le d$ with $\varphi(\es) = 0$,
we say that $\{\varphi(\alpha) : \alpha \in \calF(\{1, \dots, r\}), \alpha \ne \es\}$ is \emph{$\text{VIP}_{d,r}$}.
A set $E \subseteq \O_K$ is \emph{$\text{VIP}^*$} if it intersects every VIP set,
and $E$ is \emph{$\text{VIP}_{d,r}^*$} if it intersects every $\text{VIP}_{d,r}$ set.
Finally, $E$ is \emph{$\text{VIP}_0^*$} if for any $d \in \N$, $E$ is $\text{VIP}_{d,r}^*$ for some $r \in \N$.

As we will see below, $\text{VIP}_0^*$ is an appropriate notion of largeness for nilsequences.
However, for a multi-correlation sequence, which differs from a nilsequence by a nullsequence
(see Theorem \ref{thm: nilsequence decomp} below), we need the slightly weaker notion of $\text{AVIP}_0^*$.
A set $E$ is \emph{almost-$\text{VIP}_0^*$}, or \emph{$\text{AVIP}_0^*$} for short, if there is a $\text{VIP}_0^*$ set $A$
such that $d^*(A \setminus E) = 0$.

For any notion of largeness discussed so far, we use the added decoration of $+$ in the subscript to indicate a shift.
In particular, $\text{(A)VIP}_{0,+}^*$ means a shift of an $\text{(A)VIP}_0^*$ set.

%%%%%%%%%%%%%%%%%%%%%%%%%%%%%%%%%%%%%%%%%%%%%%%%%%%%

\subsection{Recurrence in nilmanifolds}

\begin{thm}[\cite{bl-IP}, Theorem 0.6] \label{thm: VIP returns}
	Let $(X, T)$ be a $\Z^d$-nilsystem.
	Then, for any $x_0 \in X$ and any neighborhood $U$ of $x_0$, the set
	\begin{equation*}
		R_U(x_0) := \left\{ n \in \Z^d : T^nx_0 \in U \right\}
	\end{equation*}
	is a $\text{VIP}_0^*$ set.
\end{thm}

\begin{cor} \label{cor: VIP returns}
	Let $\varphi : \Z^d \to \R$ be a nilsequence.
	For any $c < \sup{\varphi}$, the set
	\begin{equation*}
		R := \left\{ n \in \Z^d : \varphi(n) > c \right\}
	\end{equation*}
	is $\text{VIP}_{0,+}^*$.
\end{cor}
\begin{proof}
	Let $\eps = \sup{\varphi} - c > 0$.
	Then let $(X, T)$ be a minimal nilsystem, $x_0 \in X$, and $F \in C(X)$ such that
	$\sup_{n \in \Z^d}{\left| \varphi(n) - F(T^nx_0) \right|} < \frac{\eps}{2}$.
	Note that $\sup{F} > \sup{\varphi} - \frac{\eps}{2}$.
	
	Let $U := \left\{ x \in X : F(x) > \sup{\varphi} - \frac{\eps}{2} \right\}$.
	Then $U$ is a nonempty open set.
	Since $(X, T)$ is minimal, we have $T^mx_0 \in U$ for some $m \in \Z^d$.
	By Theorem \ref{thm: VIP returns},
	\begin{equation*}
		S := \left\{ n \in \Z^d : T^n(T^mx_0) \in U \right\}
	\end{equation*}
	is $\text{VIP}_0^*$.
	
	Suppose $n \in S$.
	Then
	\begin{equation*}
		\varphi(n + m) > F(T^{n+m}x_0) - \frac{\eps}{2} > \sup{\varphi} - \eps = c
	\end{equation*}
	Therefore, $R \supseteq S + m$ is $\text{VIP}_{0,+}^*$.
\end{proof}

%%%%%%%%%%%%%%%%%%%%%%%%%%%%%%%%%%%%%%%%%%%%%%%%%%%%

\subsection{Nilsequence-nulsequence decomposition}

Let $r \in \N$.
A \emph{basic $r$-step nilsequence} is a function $\varphi(n) = F(T^nx_0)$,
where $(X, \B, \mu, T)$ is an $r$-step nilsystem, $F : X \to \C$ is a continuous function, and $x_0 \in X$.
An \emph{$r$-step nilsequence} is a uniform limit of basic $r$-step nilsequences.
Knowing that a nilfactor is characteristic for polynomial multiple ergodic averages
gives a decomposition of the corresponding multi-correlation sequences.
Recall that a function $\psi : \O_K \to \C$ is a \emph{nullsequence} if $\UClim_{n \in \O_K}{|\psi(n)|^2} = 0$.

\begin{thm} \label{thm: nilsequence decomp}
	Let $K$ be a number field with ring of integers $\O_K$.
	Let $p_1, \dots, p_k \in K[x]$ be non-constant, essentially distinct, $\O_K$-valued polynomials.
	Then for any ergodic measure-preserving $\O_K$-system $\left( X, \B, \mu, T \right)$
	and any $f_0, f_1, \dots, f_k \in L^{\infty}(\mu)$, there is a decomposition
	\begin{equation*}
		a(n) := \int_X{f_0 \cdot T^{p_1(n)}f_1 \cdot{}\dots{}\cdot T^{p_k(n)}f_k~d\mu} = \varphi(n) + \psi(n),
	\end{equation*}
	where $\varphi$ is a nilsequence and $\psi$ is a nullsequence.
\end{thm}
\begin{proof}
	First, by \cite[Theorem 5.2]{br}, there exists $r \in \N$ such that
	\begin{equation*}
		a(n)
		 - \int_X{\E{f_0}{\calZ_r} \cdot T^{p_1(n)}\E{f_1}{\calZ_r} \cdot{}\dots{}\cdot T^{p_k(n)}\E{f_k}{\calZ_r}~d\mu}
	\end{equation*}
	is a nullsequence, so we may assume that $(X, \B, \mu, T)$ is a nilsystem.
	
	Next, up to a uniform approximation in $n$, we may assume that $f_0, f_1, \dots, f_k$ are continuous functions.
	Then by \cite[Theorem 1.3]{leib-nil}, $a(n)$ is the sum of a (basic) nilsequence and a nullsequence.
	Taking a uniform limit gives the desired decomposition.
\end{proof}

\begin{prop} \label{prop: syndetic implies AVIP}
	Let $K$ be a number field with ring of integers $\O_K$.
	Suppose $\varphi : \O_K \to \C$ is a nilsequence, $\psi : \O_K \to \C$ is a nullsequence,
	and $a(n) = \varphi(n) + \psi(n)$.
	Suppose that for some $c > 0$, the set
	\begin{equation*}
		R(c) := \left\{ n \in \O_K : a(n) > c \right\}
	\end{equation*}
	is syndetic.
	Then $R(c')$ is $\text{AVIP}_{0,+}^*$ for every $c' < c$.
\end{prop}
\begin{proof}
	Let $c' < c$.
	Then the set
	\begin{equation*}
		E := \left\{ n \in \O_K : |\psi(n)| \ge \frac{c - c'}{2} \right\}
	\end{equation*}
	has upper Banach density $d^*(E) = 0$.
	Therefore, $R(c) \setminus E$ is still syndetic; in particular, it is nonempty.
	But for $n \in R(c) \setminus E$, we have $\varphi(n) > c - \frac{c-c'}{2} = \frac{c+c'}{2}$.
	So, by Corollary \ref{cor: VIP returns},
	\begin{equation*}
		S := \left\{ n \in \O_K : \varphi(n) > \frac{c+c'}{2} \right\}
	\end{equation*}
	is $\text{VIP}_{0,+}^*$.
	Finally, since $\frac{c+c'}{2} - \frac{c-c'}{2} = c'$,
	we have $R(c') \supseteq S \setminus E$, so $R(c')$ is $\text{AVIP}_{0,+}^*$.
\end{proof}

By Theorem \ref{thm: nilsequence decomp}, Proposition \ref{prop: syndetic implies AVIP}
applies to polynomial multi-correlation sequences in ergodic systems.
We can therefore strengthen the conclusions of Theorems \ref{thm: IndLargeInt} and \ref{thm: MultLargeInt},
respectively, under the assumption of ergodicity:

\begin{thm} \label{thm: IndLargeInt AVIP}
	Let $K$ be a number field with ring of integers $\O_K$.
	Suppose $\{p_1, \dots, p_k\} \subseteq K[x]$ is a jointly intersective family
	of linearly independent $\O_K$-valued polynomials.
	Then for any ergodic measure-preserving $\O_K$-system $\left( X, \B, \mu, T \right)$,
	$A \in \B$, and $\eps > 0$, the set
	\begin{equation*}
		\left\{ n \in \O_K :
		 \mu \left( A \cap T^{-p_1(n)}A \cap \cdots \cap T^{-p_k(n)}A \right) > \mu(A)^{k+1} - \eps \right\}
	\end{equation*}
	is $\text{AVIP}_{0,+}^*$.
\end{thm}

\begin{thm} \label{thm: MultLargeInt AVIP}
	Let $K$ be a number field with ring of integers $\O_K$.
	Let $p(x) \in K[x]$ be an $\O_K$-valued intersective polynomial.
	Let $r,s \in \O_K$ be distinct and nonzero.
	Then for any ergodic measure-preserving $\O_K$-system $\left( X, \B, \mu, T \right)$,
	$A \in \B$, and $\eps > 0$, the set
	\begin{equation*}
		\left\{ n \in \O_K :
		 \mu \left( A \cap T^{-rp(n)}A \cap T^{-sp(n)}A \right) > \mu(A)^3 - \eps \right\}
	\end{equation*}
	is $\text{AVIP}_{0,+}^*$.
	
	Moreover, if $\frac{s}{r} \in \Q$, then
	\begin{equation*}
		\left\{ n \in \O_K :
		 \mu \left( A \cap T^{-rp(n)}A \cap T^{-sp(n)}A \cap T^{-(r+s)p(n)}A \right) > \mu(A)^4 - \eps \right\}
	\end{equation*}
	is $\text{AVIP}_{0,+}^*$.
\end{thm}

%%%%%%%%%%%%%%%%%%%%%%%%%%%%%%%%%%%%%%%%%%%%%%%%%%%%
%%%%%%%%%%%%%%%%%%%%%%%%%%%%%%%%%%%%%%%%%%%%%%%%%%%%

\section*{Acknowledgements}

The authors thank Jonathan Lubin for providing a key idea in the proof of Proposition \ref{prop: ind coord}.

%%%%%%%%%%%%%%%%%%%%%%%%%%%%%%%%%%%%%%%%%%%%%%%%%%%%
%%%%%%%%%%%%%%%%%%%%%%%%%%%%%%%%%%%%%%%%%%%%%%%%%%%%

\end{document}